\documentclass[11pt]{amsart}
\expandafter\providecommand\csname subjclassname@2020\endcsname{%
  \textup{2020} Mathematics Subject Classification}
\usepackage{amsmath,amssymb,amsthm,mathtools,bm}
\usepackage{booktabs}
\usepackage{tikz}
\usetikzlibrary{arrows.meta,calc,positioning}
\usepackage{xurl}
\usepackage[colorlinks=true,linkcolor=blue!55!black,citecolor=blue!55!black,urlcolor=blue!55!black]{hyperref}
\usepackage[capitalise,nameinlink,noabbrev]{cleveref}
\hypersetup{pdftitle={The strong type endpoint estimate for the local spherical maximal function},pdfauthor={Sanghyuk Lee and Shuijiang Zhao},pdfsubject={Mathematics Subject Classification 42B25, 35L05},pdfkeywords={spherical maximal function, wave equation, bilinear restriction, decoupling, endpoint estimates}}
\numberwithin{equation}{section}
\newtheorem{theorem}{Theorem}[section]
\newtheorem{proposition}[theorem]{Proposition}
\newtheorem{lemma}[theorem]{Lemma}
\newtheorem{corollary}[theorem]{Corollary}
\newtheorem*{citedtheorem}{Theorem}
\theoremstyle{remark}
\newtheorem{remark}[theorem]{Remark}

\newcommand{\R}{\mathbb R}
\newcommand{\Z}{\mathbb Z}
\newcommand{\Sph}{\mathbb S}
\newcommand{\supp}{\operatorname{supp}}
\newcommand{\dist}{\operatorname{dist}}
\newcommand{\eps}{\varepsilon}
\newcommand{\M}{\mathcal M}

\title[The spherical maximal function at the endpoint]{The strong type endpoint estimate for the local spherical maximal function}
\author{Sanghyuk Lee}
\address{Department of Mathematical Sciences and RIM,\newline
Seoul National University, Seoul 08826, Republic of Korea}
\email{shklee@snu.ac.kr}
\author{Shuijiang Zhao}
\address{Research Institute of Mathematics,\newline
Seoul National University, Seoul 08826, Republic of Korea}
\email{shuijiangzhao25@snu.ac.kr}
\begin{document}
\begin{abstract}
Let $p_d=(d^2+1)/(d(d-1))$ and $q_d=(d^2+1)/(d-1)$. We prove that the
local spherical maximal operator $\sup_{1\le t\le2}|f*d\sigma_t|$ is
bounded from $L^{p_d}(\R^d)$ to $L^{q_d}(\R^d)$ for every $d\ge2$.
Previously only the restricted weak type estimate was known at this
endpoint. Together with known results, this completes the
characterization of the strong type $L^p\to L^q$ bounds for the local
spherical maximal operator. We also prove the maximal estimate
\[
 \Big\|\sup_{t\in\R}|e^{it\sqrt{-\Delta}}f|\Big\|_{L^{2(d+1)/(d-1)}(\R^d)}
 \lesssim\|f\|_{\dot H^{d/(d+1)}(\R^d)}
\]
for the wave equation at the critical exponent. Both results are
obtained by exploiting the gain in bilinear estimates that comes from
the interaction between high and low frequencies.
\end{abstract}
\maketitle
\section{Introduction}\label{sec:introduction}

Let $d\sigma$ be the normalized surface measure on the unit sphere
$\Sph^{d-1}\subset\R^d$, $d\ge2$, and set
\[
 \M f(x)=\sup_{1\le t\le2}|A_t f(x)|, 
\]
where 
\[ A_t f(x)=\int_{\Sph^{d-1}}f(x-t\omega)\,d\sigma(\omega).\]
Spherical maximal operators arose in connection with the
differentiation of averages over hypersurfaces. Stein
\cite{Stein1976} proved that the full maximal operator
$\sup_{t>0}|A_t f|$ is bounded on $L^p(\R^d)$ for $p>d/(d-1)$ when
$d\ge3$, and Bourgain \cite{Bourgain1986} proved the corresponding
result for $p>2$ in the plane. Both ranges are sharp. At the critical
exponent $p=d/(d-1)$, Bourgain \cite{Bourgain1985} obtained a
restricted weak type estimate for $d\ge3$, whereas the analogous
estimate fails in the plane \cite{STW2003}.

When the radii are restricted to a compact interval, the maximal
operator improves integrability, and one asks for which pairs $(p,q)$
\begin{equation}\label{intro:eq:bound}
 \|\M f\|_{L^q(\R^d)}\le C\|f\|_{L^p(\R^d)}.
\end{equation}
Throughout the paper we write
\[
 p_d=\frac{d^2+1}{d(d-1)},\qquad q_d=\frac{d^2+1}{d-1},
\]
and we denote by $\mathcal H_d$ the convex hull of the points
\[
 P_1=(0,0),\quad
 P_2=\Big(\frac{d-1}{d},\frac{d-1}{d}\Big),\quad
 P_3=\Big(\frac{d-1}{d},\frac1d\Big),\quad
 P_4=\Big(\frac1{p_d},\frac1{q_d}\Big).
\]
For $d=2$ we have $P_2=P_3$, so $\mathcal H_2$ is a triangle.
Schlag \cite{Schlag1997} and Schlag--Sogge \cite{SchlagSogge1997}
showed that, for $p<q$, \eqref{intro:eq:bound} holds if $(1/p,1/q)$
lies in the interior of $\mathcal H_d$ and fails if $(1/p,1/q)$ lies
outside $\mathcal H_d$. 

The first author \cite{Lee2003} obtained estimates on the boundary of
$\mathcal H_d$ and proved, in particular, the restricted weak type
estimate
\[
 \|\M f\|_{L^{q,\infty}(\R^d)}\lesssim\|f\|_{L^{p,1}(\R^d)}
\]
for $(1/p,1/q)=P_4$, and also for $(1/p,1/q)=P_3$ when $d\ge3$.
Interpolating these bounds with the results mentioned above gives
strong type estimates on $\mathcal H_d\setminus([P_2,P_3]\cup\{P_4\})$ when
$d\ge3$; for $d=2$, the strong type estimates on the open segment
$(P_4,P_2)$ are contained in \cite[Theorem~1.1]{Lee2003}. Strong type
fails on the whole segment $[P_2,P_3]$, as Stein's example
\cite{Stein1976} shows; see also \cite[Theorem~4.3]{HKLL2026}.

Our main result is the strong type estimate at $P_4$.
\begin{theorem}\label{thm:main}
For every integer $d\ge2$,
\begin{equation}\label{eq:main}
 \|\M f\|_{L^{q_d}(\R^d)}\lesssim_d\|f\|_{L^{p_d}(\R^d)}.
\end{equation}
\end{theorem}
Together with the estimates described above, \cref{thm:main} gives the
complete characterization of the $L^p$--$L^q$ boundedness of $\M$,
illustrated in \cref{fig:spherical-range}. For $d=2$ we interpret
$[P_2,P_3]$ as $\{P_2\}$. Testing on sums of widely separated
translates of a fixed nonnegative compactly supported function also
shows that $q\ge p$ is necessary.

\begin{corollary}\label{cor:full-range}
Let $d\ge2$ and $1\le p,q\le\infty$. Then
\[
 \|\M f\|_{L^q(\R^d)}\lesssim\|f\|_{L^p(\R^d)}
 \quad\Longleftrightarrow\quad
 (1/p,1/q)\in\mathcal H_d\setminus[P_2,P_3].
\]
\end{corollary}
For $d\ge3$, $\M$ is of restricted weak type $(p,q)$ whenever
$(1/p,1/q)$ lies on the excluded segment $[P_2,P_3]$; see
\cite[Theorem~4.3]{HKLL2026}. For $d=2$, $\M$ is not of restricted
weak type $(2,2)$, as follows from \cite[Proposition~1.5]{STW2003}
applied to the set of radii $[1,2]$.

\begin{figure}[htbp]
\centering
\begin{minipage}[t]{0.47\textwidth}
\centering
\begin{tikzpicture}[x=5.6cm,y=5.6cm,>=Stealth,font=\small]
  \coordinate (O) at (0,0);
  \coordinate (Q) at (.5,.5);
  \coordinate (P) at (.4,.2);
  \fill[blue!7] (O)--(Q)--(P)--cycle;
  \draw[->,gray!75] (-.045,0)--(.64,0) node[right] {$1/p$};
  \draw[->,gray!75] (0,-.045)--(0,.64) node[above] {$1/q$};
  \draw[blue!70!black,very thick] (O)--(Q)--(P)--cycle;
  \fill[blue!70!black] (O) circle[radius=1.8pt];
  \draw[red!75!black,thick,fill=white] (Q) circle[radius=2.6pt];
  \fill[orange!85!black] (P) circle[radius=2.7pt];
  \node[below left] at (O) {$P_1$};
  \node[above right,align=left] at (Q) {$P_2=P_3$};
  \node[below right] at (P) {$P_4$};
  \node at (.23,.17) {$\mathcal H_2$};
  \node[anchor=north west] at (.025,.63) {$d=2$};
\end{tikzpicture}
\end{minipage}\hfill
\begin{minipage}[t]{0.49\textwidth}
\centering
\begin{tikzpicture}[x=4.8cm,y=4.8cm,>=Stealth,font=\small]
  \coordinate (O) at (0,0);
  \coordinate (Ptwo) at ({2/3},{2/3});
  \coordinate (Pthree) at ({2/3},{1/3});
  \coordinate (Pfour) at ({3/5},{1/5});
  \fill[blue!7] (O)--(Ptwo)--(Pthree)--(Pfour)--cycle;
  \draw[->,gray!75] (-.045,0)--(.81,0) node[right] {$1/p$};
  \draw[->,gray!75] (0,-.045)--(0,.81) node[above] {$1/q$};
  \draw[blue!70!black,very thick] (Ptwo)--(O)--(Pfour)--(Pthree);
  \draw[red!75!black,very thick,densely dotted] (Ptwo)--(Pthree);
  \fill[blue!70!black] (O) circle[radius=1.8pt];
  \draw[red!75!black,thick,fill=white] (Ptwo) circle[radius=2.6pt];
  \draw[red!75!black,thick,fill=white] (Pthree) circle[radius=2.6pt];
  \fill[orange!85!black] (Pfour) circle[radius=2.7pt];
  \node[below left] at (O) {$P_1$};
  \node[above right] at (Ptwo) {$P_2$};
  \node[right] at (Pthree) {$P_3$};
  \node[below right] at (Pfour) {$P_4$};
  \node at (.39,.25) {$\mathcal H_3$};
  \node[anchor=north west] at (.025,.80) {$d=3$};
\end{tikzpicture}
\end{minipage}
\caption{\small The strong type range of $\M$ in the coordinates
$(1/p,1/q)$ for $d=2$ and $d=3$: the shaded region together with the
solid boundary segments. The strong type estimate at $P_4$ is
established in \cref{thm:main}. The open circles and the dotted
segment are excluded. For $d=3$,
restricted weak type holds on the closed dotted segment; for $d=2$, it
fails at $P_2$.}
\label{fig:spherical-range}
\end{figure}
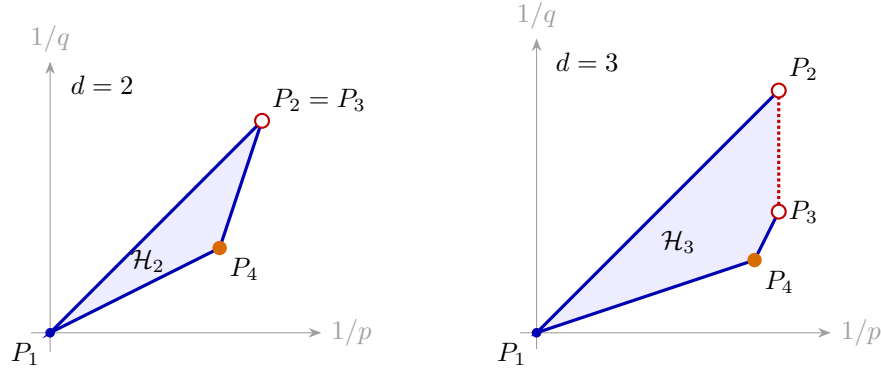

\subsubsection*{The wave maximal operator}
Our method also gives an endpoint estimate for the maximal function
associated with the wave equation. Let
\[
 q_c=\frac{2(d+1)}{d-1},\qquad s_c=\frac d2-\frac d{q_c}=\frac d{d+1}.
\]
Rogers--Villarroya \cite[Theorem~1]{RogersVillarroya2008} identified the
critical regularity for the $H^s$--$L^q$ estimates of
$\sup_{0<t<1}|e^{it\sqrt{-\Delta}}f|$ and proved the estimates above that
regularity. Cho--Lee--Li \cite[Theorem~1.1]{CLL2026} attained the
critical regularity for $2<q<\infty$, $q\ne q_c$, and obtained a weak
type estimate in Sobolev--Lorentz spaces at $q=q_c$.
The strong type estimate at $(q_c,s_c)$ has remained open;
the problem is also mentioned in \cite{KLS2025}.
\begin{theorem}\label{thm:wave}
For every $d\ge2$,
\begin{equation}\label{eq:wave-main}
 \Big\|\sup_{t\in\R}|e^{it\sqrt{-\Delta}}f|\Big\|_{L^{q_c}(\R^d)}
 \lesssim_d\|f\|_{\dot H^{s_c}(\R^d)}.
\end{equation}
\end{theorem}

\subsubsection*{Related work}
The connection between spherical maximal functions and local smoothing
estimates for the wave equation goes back to Sogge \cite{Sogge1991}
and Mockenhaupt--Seeger--Sogge \cite{MSS1992,MSS1993}. The sharp square
function estimate of Guth--Wang--Zhang \cite{GWZ2020} for the cone in
$\R^3$ gives the sharp local smoothing estimate in $\R^{2+1}$.
For wave equations with potentials in $\R^3$, Beceanu--Goldberg
\cite{BeceanuGoldberg2014} established reversed space-time estimates
and maximal estimates.

$L^p$ improving estimates have also been studied for maximal operators
over restricted sets of radii. Anderson--Hughes--Roos--Seeger
\cite{AHRS2021} and Roos--Seeger \cite{RoosSeeger2023} considered
fractal sets of radii, and Lee--Roncal--Zhang--Zhao \cite{LRZZ2025}
obtained related endpoint estimates in the plane; their restricted weak
type result also covers the case where the set of radii is $[1,2]$.
Variation estimates \cite{BORSS2022} and sparse bounds
\cite{Lacey2019} also rely on sharp $L^p\to L^q$ estimates for local
spherical maximal operators. We refer to the survey of Roos--Seeger
\cite{RoosSeeger2025} for these and related problems.

Liu--Shen--Song--Yan \cite{LSSY2026} studied $L^p\to L^q$ estimates
for Stein's spherical maximal operators of complex order. Their
sufficient conditions involve strict inequalities in the real part of
the order. Our method gives strong type estimates at a family of boundary
points not covered by their strict sufficient conditions;
see \cref{rem:analytic}.

\subsubsection*{Outline of the proof}
Both theorems are derived from bilinear maximal estimates with decay
in the ratio of the input frequencies. Let $|D|=\sqrt{-\Delta}$ and
$U(t)=e^{it|D|}$. For $\lambda\ge\mu\ge2$ let
$\rho=\mu/\lambda$ and $r_d=(d+1)/(d-1)$, and let $f_\lambda$ and
$g_\mu$ have Fourier supports in the annuli $\{|\xi|\simeq\lambda\}$
and $\{|\xi|\simeq\mu\}$. We show that
\begin{equation}\label{intro:eq:wave-gap}
 \left\|\sup_{1\le t\le2}
 |U(t)f_\lambda\,U(t)g_\mu|\right\|_{r_d}
 \lesssim_d (\lambda\mu)^{d/(d+1)}\rho^{\gamma_d}
 \|f_\lambda\|_2\|g_\mu\|_2
\end{equation}
for some $\gamma_d>0$. The factor $\rho^{\gamma_d}$ is the gain coming
from the interaction between the high frequency $\lambda$ and the low
frequency $\mu$, and it allows us to sum over dyadic frequencies at
the critical exponent. A gain of this type, from the interaction
between pieces localized at two different dyadic scales, was used by
Jeong--Lee--Ryu for Hermite eigenfunctions
\cite{JeongLeeRyu2026,JeongLeeRyu2026Planar}; there the scales measure
the distance to the turning sphere rather than frequencies. By
scaling, \eqref{intro:eq:wave-gap} extends to the supremum over all
$t\in\R$ and to all frequencies $\lambda\ge\mu>0$, and summation over
dyadic frequencies proves \cref{thm:wave}.

The starting point is Tao's bilinear restriction theorem for the cone
for waves of different frequencies \cite[Theorem~1.1]{Tao2001}; see
also Wolff \cite{Wolff2001}. A Lorentz transformation and scaling give
an estimate for pairs of caps separated by an angle $\simeq\theta$,
with explicit dependence on $\theta$, $\lambda$, and $\mu$
(\cref{lem:cone-restriction}). We then follow the scheme of
\cite[Section~2]{Lee2003}: we estimate a single pair of separated
caps, sum over the pairs at a fixed angular scale, and finally sum over
the angular scales. For $d\ge3$ we interpolate the estimate for a pair
of caps with the trivial bound
$\|Uf\,Ug\|_\infty\le\|\widehat f\|_1\|\widehat g\|_1$. The sum over
angular scales is then controlled by a square summation lemma
(\cref{lem:angular-sum}), which follows from the real interpolation
identity $(L^r,L^{q_*})_{\vartheta,2}=L^2$. To take the supremum in
$t$, we decompose the larger annulus into shells of width $\mu$ and use
Bernstein's inequality in $t$; this gives \eqref{intro:eq:wave-gap}.

In the plane this interpolation does not produce decay in $\rho$
(\cref{rem:planar-obstruction}). Instead, we divide each cap of width
$\theta$ into subcaps of width $\rho^{1/2}\theta$. After a Lorentz
rescaling, the space-time Fourier support of the product of the
high-frequency wave on a subcap and the low-frequency wave on the
paired cap lies in a sector of dimensions $1\times\rho^{1/2}\times\rho$
of the $O(\rho)$ neighborhood of the cone. Hence the $\ell^2$
decoupling theorem for the cone of Bourgain--Demeter \cite{BD2015}
applies (\cref{lem:product-decoupling}), which leads to estimates at a
fixed angular scale. Summing over the angular scales as in
\cite[Lemma~2.6]{Lee2003} gives weak type estimates, and multilinear
real interpolation \cite{GLLZ2012} gives the strong bilinear estimate
\eqref{eq:planar-model-gap} for $L^{5/2}$ inputs, which is used for
\cref{thm:main} in the plane. The same argument at other exponents,
followed by interpolation with the trivial $L^1\times L^1\to L^\infty$
estimate, gives \eqref{intro:eq:wave-gap} for $d=2$.

For the spherical averages we use the asymptotic expansion of
$\widehat\sigma$, which expresses the high-frequency part of $A_t$ in
terms of half-wave propagators with symbols of order $-(d-1)/2$. For
$d\ge3$ we interpolate \eqref{intro:eq:wave-gap} with the trivial
$L^1\times L^1\to L^\infty$ estimate for spherical averages, and for
$d=2$ we use \eqref{eq:planar-model-gap} directly.
In both cases we obtain
\begin{equation}\label{intro:eq:spherical-gap}
 \left\|\sup_{1\le t\le2}|A_t f_j\,A_t g_k|\right\|_{q_d/2}
 \lesssim_d 2^{-\delta_d|j-k|}
 \|f_j\|_{p_d}\|g_k\|_{p_d},\qquad \delta_d>0,
\end{equation}
where $f_j$ and $g_k$ have Fourier supports in
$\{|\xi|\simeq2^j\}$ and $\{|\xi|\simeq2^k\}$, $j,k\ge1$; the factor
$2^{-\delta_d|j-k|}$ is the corresponding gain from the interaction
between high and low frequencies.
\Cref{thm:main} then follows from a summation lemma
(\cref{lem:power-sum}): we expand the $m$-th power of the sum over
frequencies and apply \eqref{intro:eq:spherical-gap} to the factors
with the smallest and the largest frequencies. The lemma requires
$\max\{2,p_d\}\le m\le q_d$, so we take $m=2$ for $d\ge3$ and $m=3$ for
$d=2$, where $p_2=5/2$.

\subsection*{Organization}
In \cref{sec:prelim} we collect elementary estimates, including the
summation lemma used in the proof of \cref{thm:main}. In
\cref{sec:cone} we state Tao's theorem and derive the rescaled
estimate for separated caps. The bilinear estimate
\eqref{intro:eq:wave-gap} is proved for $d\ge3$ in \cref{sec:higher}
and for $d=2$ in \cref{sec:planar}, where we also prove
\eqref{eq:planar-model-gap}, which is needed for the spherical maximal
operator in the plane. \Cref{thm:wave} is proved in \cref{sec:wave},
and \cref{thm:main} is proved in \cref{sec:spherical}.

\subsection*{Notation}
Let $\psi\in C_c^\infty((1/2,2))$ be such that
$\sum_{j\in\Z}\psi(2^{-j}s)=1$ for $s>0$. For $\ell>0$, $P_\ell$
denotes the Fourier multiplier with symbol $\psi(|\xi|/\ell)$; when an
input is already localized to such an annulus, we also use $P_\ell$
for a slightly wider cutoff that equals one on its Fourier support.
Here $\ell$ need not be dyadic; dyadic values $\ell=2^j$ appear only in
Littlewood--Paley decompositions, where we write $P_{2^j}$. A
subscript $\ell$ on a function, as in $f_\ell$, indicates that its
Fourier transform is supported in $\{\ell/2\le|\xi|\le2\ell\}$; for an
integer index $j$, $f_j$ refers to the frequency $2^j$. We write
$A\lesssim B$ if $A\le CB$ with a constant $C$ independent of the
frequencies, the angular parameters, and the number of terms in the
sums involved, and $A\simeq B$ if $A\lesssim B\lesssim A$; a subscript,
as in $\lesssim_d$, indicates the parameters on which $C$ depends.
We use the homogeneous Sobolev norm
$\|f\|_{\dot H^s}=\||\xi|^s\widehat f(\xi)\|_2$.
For the endpoint estimates proved below, we first work with Schwartz
functions or finite sums of frequency localized Schwartz functions,
and then pass to the corresponding completions.
\section{Preliminary estimates}\label{sec:prelim}
We collect three elementary lemmas: Bernstein's inequality in the time
variable, an almost orthogonality estimate for functions with Fourier
supports in rectangles, and a lemma for summing bilinear estimates
with decay in the ratio of frequencies. We also record two standard
reductions, concerning the interpolation of maximal estimates and the
signs of the half-wave propagators.

\begin{lemma}[Bernstein's inequality in time]\label{lem:modulated-bernstein}
Let $1\le r<\infty$, $B,C>0$, and $a\in\R$. Suppose that
$H\in L^r(\R^{d+1})$ and, for almost every $x$, the Fourier transform
of $t\mapsto H(x,t)$ is supported in $[a-CB,a+CB]$.
Taking the supremum in $t$ for the continuous representative, we have
\[
 \|\sup_t|H(\cdot,t)|\|_{L^r}\lesssim_{C,r}B^{1/r}\|H\|_{L^r_{x,t}},
\]
with a constant independent of $a$ and $B$.
\end{lemma}
\begin{proof}
Choose a Schwartz function $K$ with $\widehat K=1$ on $[-C,C]$, and set
$K_B(t)=BK(Bt)$. Then $e^{-iat}H=K_B*_t(e^{-iat}H)$, and H\"older's
inequality in $t$ gives
$\sup_t|H(x,t)|\le\|K_B\|_{r'}\|H(x,\cdot)\|_r$ for each $x$. Since
$\|K_B\|_{r'}=B^{1/r}\|K\|_{r'}$, the estimate follows.
\end{proof}
\begin{lemma}[Almost orthogonality]\label{lem:rectangles}
Let $\{R_a\}$ be a family of rectangles in the frequency space $\R^d$
whose dilates $CR_a$, for a fixed $C>1$, have bounded overlap. Let $P_a$ and $\widetilde P_a$ be
smooth Fourier multipliers supported in $CR_a$, with $\widetilde P_a=1$
on $R_a$, whose kernels are uniformly bounded in $L^1$. If, for each
$a$, the Fourier transform of $H_a$ in $x$ is supported in $R_a$, then
\begin{equation}\label{eq:synthesis}
 \Big\|\sum_aH_a\Big\|_{L^r_{x,t}}
 \lesssim_r\Big(\sum_a\|H_a\|_{L^r_{x,t}}^v\Big)^{1/v},
 \qquad v=\min\{r,r'\},\quad 1\le r\le\infty.
\end{equation}
Moreover, for $2\le p\le\infty$,
\begin{equation}\label{eq:analysis}
 \Big(\sum_a\|P_af\|_p^p\Big)^{1/p}\lesssim_p\|f\|_p.
\end{equation}
The constants may depend on the overlap bound and the uniform kernel
$L^1$ bounds, but not on the number, locations, or sizes of the rectangles.
\end{lemma}
\begin{proof}
By Plancherel's theorem and the bounded overlap, the map
$(h_a)\mapsto\sum_a\widetilde P_ah_a$ is bounded from $\ell^2(L^2)$ to
$L^2$, and by the uniform kernel bounds it is bounded from
$\ell^1(L^r)$ to $L^r$ for $r=1,\infty$. Interpolation and the identity
$H_a=\widetilde P_aH_a$ give \eqref{eq:synthesis}; the multipliers act
only in $x$, so the variable $t$ is a parameter. Similarly,
\eqref{eq:analysis} follows by interpolation between the bounds
$L^2\to\ell^2(L^2)$ and $L^\infty\to\ell^\infty(L^\infty)$ for
$f\mapsto(P_af)_a$.
\end{proof}
\begin{lemma}[A summation lemma]\label{lem:power-sum}
Let $1<p,q<\infty$, and for $j\ge1$ let $T_j(t)$ be linear operators
depending on a parameter $t$. Suppose that, for all $j,k\ge1$,
\begin{equation}\label{eq:power-assumption}
 \|\sup_t|T_jf_j\, T_kf_k|\|_{q/2}
 \lesssim2^{-\delta|j-k|}\|f_j\|_p\|f_k\|_p
\end{equation}
for some $\delta>0$. If $m$ is an integer with $\max\{2,p\}\le m\le q$,
then
\begin{equation}\label{eq:power-conclusion}
 \Big\|\sup_t\Big|\sum_jT_jf_j\Big|\Big\|_q
 \lesssim\Big(\sum_j\|f_j\|_p^m\Big)^{1/m}.
\end{equation}
In particular, if $f_j=P_{2^j}f$ for a Littlewood--Paley decomposition
$\{P_{2^j}\}$, the right-hand side is bounded by $C_p\|f\|_p$.
\end{lemma}
\begin{proof}
We may assume that only finitely many $f_j$ are nonzero. Taking $j=k$
in \eqref{eq:power-assumption} gives
$\|\sup_t|T_jf_j|\|_q\lesssim\|f_j\|_p$. For fixed $x$ and $t$ we
expand $(\sum_j|T_jf_j|)^m$ and order the indices; taking the supremum
in $t$ while keeping the two extreme factors at the same time, we get
\[
 \sup_t\Big|\sum_jT_jf_j\Big|^m
 \le m!\sum_{j_1\le\cdots\le j_m}
 \sup_t|T_{j_1}f_{j_1}\,T_{j_m}f_{j_m}|
 \prod_{1<i<m}\sup_t|T_{j_i}f_{j_i}|.
\]
We take the $L^{q/m}$ norm, which is a norm since $m\le q$, and use
H\"older's inequality with exponents $q/2$ and $q$, together with
\eqref{eq:power-assumption}. With $a_j=\|f_j\|_p$ this gives
\[
 \Big\|\sup_t\Big|\sum_jT_jf_j\Big|\Big\|_q^m
 \lesssim_m\sum_{j_1\le\cdots\le j_m}
       2^{-\delta(j_m-j_1)}\prod_{i=1}^m a_{j_i}.
\]
For fixed gaps $u_i=j_{i+1}-j_i\ge0$, H\"older's inequality for
sequences bounds the sum over $j_1$ by $\sum_ja_j^m$, and summing
$2^{-\delta(u_1+\cdots+u_{m-1})}$ over the gaps gives
$(1-2^{-\delta})^{-(m-1)}$. This proves \eqref{eq:power-conclusion}.
Finally, since $m\ge p$ and $m\ge2$, Minkowski's inequality and the
Littlewood--Paley inequality give
\[
 \Big(\sum_j\|f_j\|_p^m\Big)^{1/m}
 \le\Big\|\Big(\sum_j|f_j|^m\Big)^{1/m}\Big\|_p
 \le\Big\|\Big(\sum_j|f_j|^2\Big)^{1/2}\Big\|_p\lesssim_p\|f\|_p.
 \qedhere
\]
\end{proof}

\begin{remark}[Linearization]\label{rem:linearization}
When we interpolate estimates for maximal operators, we fix a
measurable function $x\mapsto t(x)$ with values in the relevant
interval and interpolate the linear or bilinear operators obtained by
evaluating at $t=t(x)$. Since all bounds are uniform in the choice of
$t(\cdot)$, the interpolated estimate holds for the maximal operator.
Indeed, for frequency localized Schwartz inputs, one may first take
the maximum over a finite set of times and choose a measurable
maximizing time. Continuity in $t$ then recovers the supremum by
increasing finite subsets of a countable dense set.
We use this reduction without further comment.
\end{remark}

\begin{remark}[Signs]\label{rem:signs}
Let $U_\pm(t)=e^{\pm it|D|}$, so that $U=U_+$. Since $\overline{U_-(t)g}=U_+(t)\bar g$ and
$\widehat{\bar g}(\xi)=\overline{\widehat g(-\xi)}$, we have
$|U_-(t)g|=|U_+(t)\bar g|$, where the Fourier transform of $\bar g$ is
supported in the same annulus as that of $g$ and
$\|\bar g\|_p=\|g\|_p$. The same holds with $U_\pm(t)$ replaced by
$w(t)U_\pm(t)$ for a real-valued function $w$. Hence, for maximal
estimates of products $|U_\pm(t)f\,U_\pm(t)g|$, with any choice of
the two signs, it suffices to consider the case in which both signs
are $+$.
\end{remark}

\section{Bilinear restriction for the cone}\label{sec:cone}
In this section we state Tao's bilinear restriction theorem for the
cone and rescale it to obtain an estimate for separated caps with
explicit dependence on the angle and on both frequencies
(\cref{lem:cone-restriction}). We use Tao's theorem in its scalar
form. For $k\in\Z$ let
\[
 \Sigma_k^\pm=\{(\xi,\pm|\xi|):
       2^k\le|\xi|\le2^{k+1},\ \angle(\xi,e_1)\le\pi/8\}.
\]
Following \cite{Tao2001}, a \emph{red wave of frequency $2^k$} is a
function on $\R^{d+1}$ whose space-time Fourier transform is given by
an $L^2$ density on $\Sigma_k^+$, parametrized by $\xi$; a \emph{blue
wave} is defined in the same way with $\Sigma_k^-$ in place of
$\Sigma_k^+$. The energy of such a wave $\phi$ is
$E(\phi)=\|\phi(\cdot,t)\|_2^2$, which is independent of $t$.
\begin{citedtheorem}[Tao {\cite[Theorem~1.1]{Tao2001}}]
Let $d\ge2$, let $\phi$ be a red wave of frequency $1$, and let $\psi$
be a blue wave of frequency $2^k$, $k\ge0$. Then, for
$(d+3)/(d+1)\le q\le2$ and $\eps>0$,
\begin{equation}\label{eq:tao-red-blue}
 \|\phi\psi\|_{L^q(\R^{d+1})}
 \lesssim_{d,q,\eps}
 2^{k(1/q-1/2+\eps)}E(\phi)^{1/2}E(\psi)^{1/2}.
\end{equation}
\end{citedtheorem}
As in \cite[Lemma~2.5]{Lee2003}, rescaling gives the dependence on
the angle between the caps and on both frequencies.
\begin{lemma}[Bilinear restriction for separated caps]\label{lem:cone-restriction}
Let $d\ge2$ and $\lambda\ge\mu>0$. Let $0<\theta\lesssim1$, and let
$c_0>0$ be a sufficiently small constant. Suppose that $\widehat f$ and
$\widehat g$ are supported in $\{|\xi|\simeq\lambda\}$ and
$\{|\xi|\simeq\mu\}$, respectively, and in caps of width at most
$c_0\theta$ whose centers are at an angle $\simeq\theta$. Then, for
$(d+3)/(d+1)\le q\le2$ and $\eps>0$,
\begin{equation}\label{eq:cone-restriction}
 \|Uf\,Ug\|_{L^q_{x,t}}
 \lesssim_{d,q,\eps}\theta^{d-1-(d+1)/q}\mu^{d-(d+1)/q}
       (\lambda/\mu)^{1/q-1/2+\eps}\|f\|_2\|g\|_2.
\end{equation}
\end{lemma}
\begin{proof}
Let $\delta\simeq\theta$ be the angle between the centers of the caps.
After a rotation we may assume that the centers are
\[
 (\cos(\delta/2),\pm\sin(\delta/2),0,\ldots,0).
\]
On the forward cone we use the null coordinates
\[
 u=\frac{\tau+\xi_1}{2},\qquad
 v=\frac{\tau-\xi_1}{2},\qquad z=\xi',\qquad 4uv=|z|^2.
\]
Let $a=\tan(\delta/4)\simeq\theta$, and let $A$ be the linear map of
space-time frequency space given by
\[
 (u,v,z)\longmapsto (u,a^{-2}v,a^{-1}z).
\]
The map $A$ preserves the forward cone and has determinant
$a^{-(d+1)}$. At the centers of the caps we have $z/(2u)=\pm ae_1'$,
and after applying $A$ this ratio becomes $\pm e_1'$. Hence the images
of the two caps have opposite spatial directions and width $O(c_0)$,
and their radial frequencies remain comparable to $\lambda$ and $\mu$,
uniformly in $\theta$.

In these coordinates
\[
 \xi=\left(u-\frac{|z|^2}{4u},z\right),\qquad
 d\xi=J(u,z)\,du\,dz,\qquad
 J(u,z)=1+\frac{|z|^2}{4u^2}.
\]
Write $Z=z/a$ and $\Xi=(u-|Z|^2/(4u),Z)$, and define $\widetilde f$ by
\[
 \widehat{\widetilde f}(\Xi)
   =\frac{J(u,aZ)}{J(u,Z)}
      \widehat f\left(u-\frac{a^2|Z|^2}{4u},aZ\right),
\]
and $\widetilde g$ similarly. Both Jacobian factors are bounded above
and below on the supports. Changing variables in the extension
integral, we obtain
\begin{equation}
 Uf\big(A^T(y,s)\big)=a^{d-1}U\widetilde f(y,s),
 \qquad
 \|\widetilde f\|_2\lesssim a^{-(d-1)/2}\|f\|_2,
 \label{eq:cone-data-rescaling}
 \end{equation}
 and
 \begin{equation}
 \|Uf\,Ug\|_q
   =a^{2(d-1)-(d+1)/q}
       \|U\widetilde f\,U\widetilde g\|_q,
 \label{eq:cone-product-rescaling}
\end{equation}
where $A$ is identified with its matrix in the coordinates
$(\xi,\tau)$.

After a rotation, the transformed low-frequency cap is centered at
$e_1$ and the high-frequency cap at $-e_1$. Replacing the
high-frequency wave by its complex conjugate, which does not change
the modulus of the product, we obtain a blue wave with spatial
direction $e_1$, while the low-frequency wave is red. For $c_0$
sufficiently small, both apertures are at most $\pi/8$. We decompose each Fourier
support into its intersections with the dyadic annuli
$\{2^j\le|\xi|\le2^{j+1}\}$, $j\in\Z$; each support meets a bounded
number of them. For a pair of pieces with indices
$j_{\mathrm{lo}}\le j_{\mathrm{hi}}$, if the piece of the original
high-frequency wave lies in the smaller annulus, then
$\lambda\simeq\mu$, and conjugating both waves and applying a common
rotation we may again assume that the wave of smaller frequency is
red. Scaling by $2^{-j_{\mathrm{lo}}}$ reduces the pair to a red wave
of frequency $1$ and a blue wave of frequency $2^k$ with
\[
 k=j_{\mathrm{hi}}-j_{\mathrm{lo}}\ge0,\qquad
 2^{j_{\mathrm{lo}}}\simeq\mu,\qquad 2^k\simeq\lambda/\mu,
\]
and, for $L^2$ normalized data, the scaling produces the factor
$2^{j_{\mathrm{lo}}(d-(d+1)/q)}\simeq\mu^{d-(d+1)/q}$. By
\eqref{eq:tao-red-blue}, summed over the bounded number of pieces,
\[
 \|U\widetilde f\,U\widetilde g\|_q
 \lesssim_{d,q,\eps}
 \mu^{d-(d+1)/q}(\lambda/\mu)^{1/q-1/2+\eps}
 \|\widetilde f\|_2\|\widetilde g\|_2.
\]
Combining this with \eqref{eq:cone-data-rescaling} and
\eqref{eq:cone-product-rescaling}, we obtain the angular factor
\[
 a^{2(d-1)-(d+1)/q}a^{-(d-1)}=a^{d-1-(d+1)/q},
\]
which proves \eqref{eq:cone-restriction}. Note that the $\eps$-loss appears only in
the frequency ratio.
\end{proof}
\begin{remark}\label{rem:candy}
For $(d+3)/(d+1)<q\le2$, the factor $(\lambda/\mu)^\eps$ in
\eqref{eq:cone-restriction} can be removed by a theorem of Candy
\cite[Theorem~1.10]{Candy2019}. Indeed, in that theorem we take both
masses equal to zero, $\alpha\simeq\theta$, $\beta=1$, and equal
space and time exponents $q$. The transversality parameter is
\[
 \left(\frac{|\xi_0||\zeta_0|-\xi_0\cdot\zeta_0}
                  {\lambda\mu}\right)^{1/2}\simeq\theta,
\]
and the condition on the exponents is $(d+3)/(2q)<(d+1)/2$. The
support conditions are satisfied after a fixed finite radial
subdivision and by the choice of $c_0$, and the resulting bound is
\[\theta^{d-1-(d+1)/q}\mu^{d-(d+1)/q}(\lambda/\mu)^{1/q-1/2}.\] We do
not need this refinement: below we use \eqref{eq:cone-restriction}
only for one exponent $q_0$ in the interior of the range when $d\ge3$,
and for some $q\in(19/10,2)$ and for $q=2$ when $d=2$ (this exponent
is denoted by $s$ in \cref{sec:planar}), with $\eps>0$ small.
\end{remark}
\section{The bilinear wave maximal estimate}\label{sec:bilinear}\label{sec:higher}
In this section we state the bilinear estimate with decay in the
frequency ratio that underlies both main theorems, and prove it for
$d\ge3$.

Let
\[
 \gamma_d^*=\begin{cases}1/12,&d=2,\\1/(d+1),&d\ge3.\end{cases}
\]
\begin{proposition}\label{prop:wave-gap}
Let $d\ge2$, $0<\gamma_d<\gamma_d^*$, $\lambda\ge\mu\ge2$, and
$\rho=\mu/\lambda$. For inputs $f_\lambda,g_\mu$ supported in the
corresponding frequency annuli,
\begin{equation}\label{eq:wave-gap-local}
 \left\|\sup_{1\le t\le2}|U(t)f_\lambda\,U(t)g_\mu|\right\|_{r_d}
 \lesssim_{d,\gamma_d}(\lambda\mu)^{s_c}\rho^{\gamma_d}
 \|f_\lambda\|_2\|g_\mu\|_2,
\end{equation}
where $r_d=(d+1)/(d-1)$ and $s_c=d/(d+1)$.
\end{proposition}
By \cref{rem:signs}, the same estimate holds with $U(t)$ replaced by
$U_\pm(t)$ in either factor, with any choice of signs. For the proofs
of \cref{thm:main,thm:wave} it suffices that \eqref{eq:wave-gap-local}
holds for some $\gamma_d>0$; the range $0<\gamma_d<\gamma_d^*$ is what
our argument gives. The case $d=2$ is proved in \cref{sec:planar-l2}.

\subsection{Angular decomposition}
We fix the exponents used in the interpolation and decompose the pair
of input frequencies according to the angle between them. Let $d\ge3$
and $0<\gamma_d<\gamma_d^*$. We choose $\chi,\eps>0$ so small that
\begin{equation}\label{eq:highdim-choice}
 \chi<\frac2{(d+1)^2},\qquad
 (d+1)\chi+\eps<\gamma_d^*-\gamma_d,
\end{equation}
and set $1/r=1/2+\chi$, $\beta=1-2\chi$, and $q_0=\beta r_d$. Then
$(d+3)/(d+1)<q_0<2$ and $1<r<2<2r_d$.
When defining the angular pieces below, we write $f=f_\lambda$ and
$g=g_\mu$.

We use a Whitney decomposition of $\Sph^{d-1}\times\Sph^{d-1}$ away from
the diagonal. At the scale $\theta=2^{-j}$, $j\ge0$, it consists of
pairs $(\omega,\omega')$ of caps of width $\simeq\theta$ whose centers
are at distance $\simeq\theta$; each $\omega$ is paired with a bounded
number of caps $\omega'$, and conversely. We write $\omega\sim\omega'$
for such pairs. The decomposition is obtained from smooth dyadic
partitions of unity on the sphere, in finitely many coordinate charts,
by the usual telescoping argument, and each of its terms is a product
of a cutoff in $\xi$ and a cutoff in $\eta$. After a fixed finite
subdivision we may assume that the caps satisfy the hypothesis of
\cref{lem:cone-restriction}. Since the diagonal has measure zero in
$\Sph^{d-1}\times\Sph^{d-1}$, the decomposition involves all scales
$j\ge0$ and has no remaining near-diagonal term.

\subsection{A square summation lemma}
The following lemma allows us to sum over the angular scales.
Related angular summation estimates appear in Ramos's refinement of
the Strichartz inequality \cite{Ramos2012}. We give a short proof
based on real interpolation.

\begin{lemma}\label{lem:angular-sum}
Let $d\ge2$, $1<r<2<q_*<\infty$, $\ell>0$, and
$v_j=\ell^d2^{-(d-1)j}$ for $j\ge0$. Suppose that, for each $j$,
$\{\mathfrak a_{j,\omega}\}_\omega$ is a family of measurable sets in
$\R^d$ with $|\mathfrak a_{j,\omega}|\lesssim v_j$ and
$\sum_{\omega}\mathbf1_{\mathfrak a_{j,\omega}}\lesssim1$, uniformly in
$j$. For $h\in L^2(\R^d)$ set
\[
 \mathfrak x_j(h;\ell)=v_j^{1/2-1/r}
       \Big(\sum_{\omega}\|h\|_{L^r(\mathfrak a_{j,\omega})}^{q_*}\Big)^{1/q_*}.
\]
Then $\sum_{j\ge0}\mathfrak x_j(h;\ell)^2\lesssim\|h\|_2^2$, with a
constant independent of $\ell$.
\end{lemma}
\begin{proof}
By the embedding $\ell^r\subset\ell^{q_*}$, and by H\"older's inequality
on each set $\mathfrak a_{j,\omega}$,
\[
 \mathfrak x_j(h;\ell)\lesssim v_j^{1/2-1/r}\|h\|_r,
 \qquad
 \mathfrak x_j(h;\ell)\lesssim v_j^{1/2-1/q_*}\|h\|_{q_*}.
\]
Let $b=1/r-1/q_*$, $\vartheta=(1/r-1/2)/b$, and $t_j=v_j^b$. Since
$\mathfrak x_j(\cdot\,;\ell)$ is subadditive, applying these bounds to
a decomposition $h=h_0+h_1$ and taking the infimum gives
\[
 \mathfrak x_j(h;\ell)\lesssim t_j^{-\vartheta}K(t_j,h),
 \qquad
 K(t,h)=\inf_{h=h_0+h_1}
                    (\|h_0\|_r+t\|h_1\|_{q_*}),
\]
where $K$ is the $K$-functional of the couple $(L^r,L^{q_*})$. The
ratio $t_{j+1}/t_j=\sigma=2^{-(d-1)b}$ does not depend on $j$. Since
$K(t,h)$ is increasing in $t$ and $K(t,h)/t$ is decreasing,
$[t_j^{-\vartheta}K(t_j,h)]^2$ is bounded by a constant times the
integral of $[t^{-\vartheta}K(t,h)]^2$ over $[t_j,t_j/\sigma]$ with
respect to $dt/t$. These intervals have disjoint interiors, hence
\[
 \sum_j\mathfrak x_j(h;\ell)^2\lesssim
 \int_0^\infty[t^{-\vartheta}K(t,h)]^2\,\frac{dt}{t}
 \simeq\|h\|_2^2.
\]
The last equivalence is the identity
$(L^r,L^{q_*})_{\vartheta,2}=L^{2}$
\cite[Chapter~5]{BerghLofstrom1976}, since
$(1-\vartheta)/r+\vartheta/q_*=1/2$. The constants depend on
$d,r,q_*$ and on the uniform measure and overlap bounds for the sets,
but not on the initial value $t_0=\ell^{db}$. In particular, they are
independent of $\ell$.
\end{proof}

\subsection{Estimates for separated caps}
We first estimate the product of two waves with Fourier supports in a
pair of separated caps, and then sum over all pairs at a fixed angular
scale and over the angular scales.

For caps $\omega\sim\omega'$ at scale $\theta$, let
$\mathfrak a_{\lambda,\omega}$ and $\mathfrak a_{\mu,\omega'}$ be the
corresponding sectors of the annuli $\{|\xi|\simeq\lambda\}$ and
$\{|\xi|\simeq\mu\}$, so that
$|\mathfrak a_{\ell,\omega}|\simeq \ell^d\theta^{d-1}$, and let
\[
 a_{\omega}=|\mathfrak a_{\lambda,\omega}|^{-\chi}
           \|\widehat f\|_{L^r(\mathfrak a_{\lambda,\omega})},\qquad
 b_{\omega'}=|\mathfrak a_{\mu,\omega'}|^{-\chi}
           \|\widehat g\|_{L^r(\mathfrak a_{\mu,\omega'})}.
\]
We interpolate \eqref{eq:cone-restriction} with $q=q_0$ and the
trivial estimate
\[\|Uf_\omega\,Ug_{\omega'}\|_{L^\infty_{x,t}}
\le\|\widehat f_\omega\|_1\|\widehat g_{\omega'}\|_1\] with weight
$\beta=q_0/r_d=1-2\chi$ on the former. The output space is then
$L^{r_d}_{x,t}$, and $\widehat f,\widehat g$ are measured in $L^r$,
since $\beta/2+1-\beta=1/r$. More explicitly,
\[
 \begin{aligned}
 \|Uf_\omega\,Ug_{\omega'}\|_{r_d}
 \lesssim{}&
 \theta^{\beta(d-1)-(d+1)/r_d}
 \mu^{\beta d-(d+1)/r_d}
 (\lambda/\mu)^{1/r_d-\beta/2+\beta\eps}\\
 &\times\|\widehat f_\omega\|_r\|\widehat g_{\omega'}\|_r.
 \end{aligned}
\]
The power of $\theta$ in this bound,
after normalization by $|\mathfrak a_{\lambda,\omega}|^{\chi}
|\mathfrak a_{\mu,\omega'}|^{\chi}$, is
$\beta(d-1)-(d+1)/r_d+2\chi(d-1)=0$. The total power of $\lambda$ and
$\mu$ is one, and the power of $\lambda$, including the factor
$(\lambda/\mu)^{\beta\eps}$ from \eqref{eq:cone-restriction}, is
\[
 \frac1{r_d}-\frac12+(d+1)\chi+\beta\eps=\frac12-\widetilde\kappa,
 \qquad
 \widetilde\kappa=\frac2{d+1}-(d+1)\chi-\beta\eps.
\]
Hence
\begin{equation}\label{eq:highdim-cap}
 \|Uf_{\omega}\,Ug_{\omega'}\|_{r_d}\lesssim(\lambda\mu)^{1/2}
 \rho^{\widetilde\kappa}a_{\omega}b_{\omega'}.
\end{equation}

We now sum over the pairs at a fixed scale $\theta$. For $\theta$
below a sufficiently small fixed constant, the paired caps lie in a
common small aperture. Thus $|\xi+\eta|\simeq\lambda$ for
$\xi\in\mathfrak a_{\lambda,\omega}$ and $\eta\in\mathfrak a_{\mu,\omega'}$,
also when $\lambda\simeq\mu$, and $\xi+\eta$ lies in a rectangle of
dimensions $O(\lambda)\times O(\lambda\theta)^{d-1}$ with long side in
the direction of $\omega$. Choose a fixed smooth radial partition
$\sum_\nu\pi_\nu(\zeta)=1$ on the possible output annulus, with
$\pi_\nu(\zeta)=q_\nu(|\zeta|/\lambda)$ supported in short radial
intervals. The multipliers $\pi_\nu(D)$ are uniformly bounded on
$L^{r_d}$. For each $\nu$, the spatial Fourier supports of
$\pi_\nu(D)(Uf_\omega\,Ug_{\omega'})$ lie in rectangles whose fixed
dilates stay away from the origin and have bounded overlap as
$\omega$ varies. Their auxiliary cutoffs have uniformly bounded
kernel $L^1$ norms by rectangular rescaling. We apply
\cref{lem:rectangles} to these projected products and sum over the
fixed finite set of $\nu$. The finitely many pairs at angles bounded
below are handled by the triangle inequality.
Since $1<r_d\le2$, \cref{lem:rectangles} and \eqref{eq:highdim-cap}
give
\[
 \begin{aligned}
 \Big\|\sum_{\omega\sim \omega'}Uf_{\omega}Ug_{\omega'}\Big\|_{r_d}
 &\lesssim(\lambda\mu)^{1/2}\rho^{\widetilde\kappa}
       \Big(\sum_{\omega\sim \omega'}(a_{\omega}b_{\omega'})^{r_d}\Big)^{1/r_d}\\
 &\lesssim(\lambda\mu)^{1/2}\rho^{\widetilde\kappa}
 \mathfrak x_j(\widehat f;\lambda)\,\mathfrak x_j(\widehat g;\mu),
 \end{aligned}
\]
where $q_*=2r_d$ in the definition of $\mathfrak x_j$. The last
inequality follows from the Cauchy--Schwarz inequality, the bounded
number of $\omega'$ with $\omega\sim\omega'$, and the bound
$(\sum_{\omega}a_{\omega}^{2r_d})^{1/(2r_d)}\lesssim
\mathfrak x_j(\widehat f;\lambda)$ and its analogue for $g$; here the
smooth cutoffs are absorbed by slightly enlarging the sectors, which
preserves bounded overlap. Applying \cref{lem:angular-sum} to
$\widehat f$ with $\ell=\lambda$ and to $\widehat g$ with $\ell=\mu$,
and summing over $j$ with the Cauchy--Schwarz inequality and
Plancherel's theorem, we obtain
\begin{equation}\label{eq:highdim-spacetime}
 \|Uf_\lambda\, Ug_\mu\|_{L^{r_d}_{x,t}}
 \lesssim(\lambda\mu)^{1/2}\rho^{\widetilde\kappa}
 \|f_\lambda\|_2\|g_\mu\|_2.
\end{equation}

Here the series over $j$ converges absolutely in $L^{r_d}_{x,t}$,
and its sum is the product $Uf_\lambda\,Ug_\mu$. Indeed,
$\widehat f_\lambda,\widehat g_\mu\in L^1$, and the partial sums of the
Whitney decomposition are uniformly bounded and converge to $1$ off
the diagonal. By dominated convergence, the corresponding sums of
products converge pointwise to $Uf_\lambda\,Ug_\mu$.

\subsection{Radial decomposition and the supremum in time}
We now deduce the maximal estimate \eqref{eq:wave-gap-local} from the
space-time estimate \eqref{eq:highdim-spacetime}.

Decompose $f_\lambda=\sum_\nu f_\nu$ smoothly into $O(\lambda/\mu)$
pieces whose Fourier transforms are supported in shells
$\{||\xi|-\varrho_\nu|\lesssim\mu\}$, with
$\sum_\nu\|f_\nu\|_2^2\lesssim\|f_\lambda\|_2^2$. Then the Fourier transform in
$t$ of $Uf_\nu\,Ug_\mu$ is supported in an interval of
length $O(\mu)$ centered at $\varrho_\nu$. By
\cref{lem:modulated-bernstein}, \eqref{eq:highdim-spacetime} for the
pairs $(f_\nu,g_\mu)$, the triangle inequality, and the
Cauchy--Schwarz inequality in $\nu$,
\begin{align}
 \|\sup_t|Uf_\lambda\, Ug_\mu|\|_{r_d}
 &\lesssim\mu^{1/r_d}(\lambda\mu)^{1/2}\rho^{\widetilde\kappa}
       (\lambda/\mu)^{1/2}\|f_\lambda\|_2\|g_\mu\|_2\nonumber\\
 &=(\lambda\mu)^{s_c}\rho^{\widetilde\kappa-1/(d+1)}
       \|f_\lambda\|_2\|g_\mu\|_2.
 \label{eq:highdim-maximal}
\end{align}
By \eqref{eq:highdim-choice},
$\widetilde\kappa-1/(d+1)=\gamma_d^*-(d+1)\chi-\beta\eps>\gamma_d$, which
proves \cref{prop:wave-gap} for $d\ge3$. In fact, \eqref{eq:highdim-maximal}
holds with the supremum taken over all $t\in\R$ and for all
$\lambda\ge\mu>0$.

\begin{remark}[The planar case]\label{rem:planar-obstruction}
The argument above does not give decay in $\rho$ when $d=2$. In this
case $r_2=3$, and the interpolation of \eqref{eq:cone-restriction}
with an exponent $q_0\le2$ requires $\chi=(1-q_0/3)/2\ge1/6$. The power
of $\rho$ remaining after the radial decomposition is then
$1/3-3\chi-\beta\eps\le-1/6$, even if $\eps=0$. In the next section we
obtain decay in $\rho$ in the plane from a finer angular decomposition
and the decoupling theorem for the cone.
\end{remark}
\section{Bilinear estimates in the plane}\label{sec:planar}

In this section we prove \cref{prop:wave-gap} for $d=2$, and the
$L^{5/2}$ estimate in \cref{prop:planar-gap}, which is used for
\cref{thm:main} in the plane. Both are based on estimates at a fixed
angular scale (\cref{lem:planar-angular-scale}); we sum them over the
angular scales and obtain strong type estimates by multilinear
interpolation, following the structure of the argument in
\cite{Lee2003}.

Throughout this section $d=2$. We denote by $\eps$ both the loss in
\eqref{eq:cone-restriction} and the loss in the decoupling
inequality; it is fixed in \cref{sec:planar-sum} for the
proof of \cref{prop:planar-gap}, and chosen small in
\cref{sec:planar-l2}. Unless stated otherwise we assume
$\lambda\ge16\mu\ge32$; the case $\lambda<16\mu$ is treated in
\cref{sec:planar-completion}. All frequency cutoffs are smooth, and the numerical constants defining
annuli and caps may be changed by fixed factors when necessary.

\subsubsection*{A time weight}
To retain exact information on the space-time Fourier support while
using estimates that are global in time, we multiply the half-wave
propagators by a band-limited weight in $t$. Fix a sufficiently small constant $b_*>0$ and a real-valued Schwartz
function $w$
on $\R$ such that
\[
 \supp\widehat w\subset[-b_*/2,b_*/2],\qquad
 w(0)=0,\qquad \inf_{1\leq t\leq2}|w(t)|>0.
\]
Then $\supp\widehat{w^2}\subset[-b_*,b_*]$. For instance, let
$w(t)=tw_0(t)$, where $w_0$ is the inverse Fourier transform of a
nonnegative even bump function with sufficiently small support; then
$w_0$ is positive on $[0,3]$. We choose $b_*$ small compared with the
constants defining the annuli and caps; this guarantees that
$u\simeq1$ in \eqref{eq:planar-lorentz-rescaling} below. Set
\[
 V_\ell f(x,t)=w(t)\ell^{-1/2}U(t)P_\ell f(x).
\]
The condition $w(0)=0$
is used only in the proof of the $L^1\to L^\infty$ estimate in
\eqref{eq:planar-cap-kernels}.

\begin{proposition}\label{prop:planar-gap}
There is $c>0$ such that
\begin{equation}
 \left\|\sup_{1\leq t\leq2}|V_\lambda f(x,t)V_\mu g(x,t)|\right\|_{5/2}
 \lesssim \rho^c\|f\|_{5/2}\|g\|_{5/2}.
 \label{eq:planar-model-gap}
\end{equation}
If $\lambda\simeq\mu$, the estimate holds with $\rho^c$
replaced by a constant.
\end{proposition}

\subsection{Angular decomposition and kernel estimates}
We decompose the inputs into pairs of caps as in \cref{sec:higher},
and record the kernel estimates for the operators $V_\ell$ restricted
to caps. We use a Whitney decomposition of $\Sph^1\times\Sph^1$. Away from a
neighborhood of the diagonal it consists of pairs $(\omega,\omega')$
of caps with
\[
 |\omega|\simeq |\omega'|\simeq \dist(\omega,\omega')\simeq\theta,
 \qquad \mu^{-1/2}\lesssim\theta\lesssim1.
\]
For each $\theta$, every $\omega$ is paired with a bounded number of
caps $\omega'$, and conversely; we write $\omega\sim\omega'$. As in
\cref{sec:higher}, such a decomposition is obtained by telescoping
products of smooth dyadic angular partitions on finitely many charts.
Each term is a product of cutoffs in the two input variables, and the
decomposition is exact, including the term near the diagonal. After a
fixed finite subdivision the widths of the caps are small compared
with their separation, as required in \cref{lem:cone-restriction}.

For the high-frequency input, $P_\omega$ denotes the angular cutoff to
$\omega$ multiplied by $\widetilde\psi(|\xi|/\lambda)$, where
$\widetilde\psi\in C_c^\infty((0,\infty))$ equals one on $\supp\psi$.
For the low-frequency input, we use the corresponding cutoff at
frequency $\mu$. We suppress the radial frequency in this notation
and write $f_\omega=P_\omega f$.
For a fixed angular scale $\theta$ set
\begin{equation}\label{eq:planar-bt}
 \begin{aligned}
 \mathfrak b_\theta(f,g)(x,t)
   &=\sum_{\omega\sim \omega'}V_\lambda f_{\omega}(x,t)\,V_\mu g_{\omega'}(x,t),\\
 \mathfrak t_\theta(f,g)(x)
   &=\sup_{1\leq t\leq2}|\mathfrak b_\theta(f,g)(x,t)|.
 \end{aligned}
\end{equation}

Let $P_{\ell,\delta}$ be a cutoff to the intersection of the annulus
$\{|\xi|\simeq\ell\}$ with a sector of angular width
$\delta\geq \ell^{-1/2}$. Then
\begin{equation}
 \label{eq:planar-cap-kernels}
\begin{aligned}
 \|w(t)\ell^{-1/2}U(t)P_{\ell,\delta}\|_{L^\infty_x\to L^\infty_{x,t}}
 \lesssim\delta,
 \\
 \|w(t)\ell^{-1/2}U(t)P_{\ell,\delta}\|_{L^1_x\to L^\infty_{x,t}}
 \lesssim \ell.
 \end{aligned}
\end{equation}
Here and below the cutoffs satisfy uniform symbol estimates after
rescaling their supporting rectangles. To prove the first estimate, we
decompose the sector into sectors of width
$\delta_t=(\ell(1+|t|))^{-1/2}$. After a rotation and the substitution
$\xi_1=\ell a$, $\xi_2=\ell\delta_t b$, and after removing the linear
phase $t\xi_1$, each piece becomes a smooth compactly supported
multiplier with uniform bounds, since $|t|\ell\delta_t^2\leq1$. Hence
each piece has a kernel of $L^1$ norm $O(\ell^{-1/2})$. There are
$O(\delta\sqrt{\ell(1+|t|)})$ pieces, so the kernel of
$\ell^{-1/2}U(t)P_{\ell,\delta}$ has $L^1$ norm
$O(\delta(1+|t|)^{1/2})$, and multiplication by $w(t)$ gives the first
estimate. For the second, stationary phase in the angular variable and
integration by parts in the radial variable away from $|x|\simeq|t|$
give the dispersive estimate
\[
 \|\ell^{-1/2}U(t)P_{\ell,\delta}\|_{L^1\to L^\infty}
 \lesssim \ell|t|^{-1/2},\qquad t\ne0;
\]
the derivatives of the angular cutoff cause no loss, since their total
variation is uniformly bounded. Since
$\sup_t |w(t)||t|^{-1/2}<\infty$, the second estimate in
\eqref{eq:planar-cap-kernels} follows.

For $d=2$, $5/3\le s\le2$, and $\eps>0$, \cref{lem:cone-restriction}
reads
\[
 \|U(t)f\,U(t)g\|_{L^s_{x,t}}
 \lesssim_{s,\eps}\lambda^{1/s-1/2+\eps}\mu^{5/2-4/s-\eps}
             \theta^{1-3/s}\|f\|_2\|g\|_2.
\]
The estimate also holds when the high-frequency cap is narrower than
$c_0\theta$. The additional narrowness will be exploited through
\eqref{eq:planar-cap-kernels} and the geometry of the Fourier support
of the product.

\subsection{Decoupling for products of half-waves}
The following lemma is the main new ingredient in the plane. It shows
that the products of the high-frequency waves on subcaps of width
$\rho^{1/2}\theta$ with a fixed low-frequency wave can be decoupled
with only a $\rho^{-\eps}$ loss.

\begin{lemma}\label{lem:product-decoupling}
Let $\lambda\ge16\mu\ge32$, $\rho=\mu/\lambda$, and
$\mu^{-1/2}\lesssim\theta\lesssim1$. Let $\omega$ and $\omega'$ be
caps of width $O(\theta)$ whose centers are at distance $O(\theta)$,
and suppose that $\omega$ has aperture at most a fixed small constant
$c_0<\pi/2$. Divide $\omega$ into subcaps $\omega_\nu$ of width
$\delta=\rho^{1/2}\theta$ whose centers are separated by a fixed
multiple of $\delta$ and which have bounded overlap. Suppose that
$\widehat f_\nu$ is supported in $\{\lambda/2\le|\xi|\le2\lambda\}$ and
in the sector determined by $\omega_\nu$, and that $\widehat g_{\omega'}$
is supported in $\{\mu/2\le|\xi|\le2\mu\}$ and in the sector determined
by $\omega'$. Let
\[
 F_\nu(x,t)=V_\lambda f_\nu(x,t)\,V_\mu g_{\omega'}(x,t).
\]
Then, for $2\le r\le6$ and $\eps>0$,
\begin{equation}\label{eq:planar-output-decoupling}
 \Big\|\sum_\nu F_\nu\Big\|_{L^r_{x,t}}
 \lesssim_\eps\rho^{-\eps}
 \Big(\sum_\nu\|F_\nu\|_{L^r_{x,t}}^2\Big)^{1/2}.
\end{equation}
The implicit constant depends only on the
support constants, $r$, $\eps$, and $w$.
\end{lemma}
No lower bound on the separation of $\omega$ and $\omega'$ is assumed;
this is used for the near-diagonal part in \cref{sec:planar-completion}.

\begin{proof}
Let $\mathbf e(\alpha)=(\cos \alpha,\sin \alpha)$. After a rotation taking the
center of $\omega$ to $e_1$, the space-time Fourier support of $F_\nu$
consists of points
\[
 (\xi,\tau)
 =\bigl(R_1\mathbf e(\alpha)+R_2\mathbf e(\beta),\,R_1+R_2+\sigma\bigr),
 \qquad |\sigma|\leq b_*,
\]
where $\supp\widehat{w^2}\subset[-b_*,b_*]$ gives the bound on
$\sigma$, and
\[
 R_1\simeq\lambda,\quad R_2\simeq\mu,\quad
 |\alpha|,|\beta|\lesssim\theta,\quad
 |\alpha-\alpha_\nu|\lesssim \delta=\rho^{1/2}\theta,
\]
with $\alpha_\nu$ the center of $\omega_\nu$. We have $|\alpha|\le c_0<\pi/2$.

If $\sigma=0$,
\[
 \tau^2-|\xi|^2
 =2R_1R_2\bigl(1-\cos(\alpha-\beta)\bigr)
 =O(\lambda\mu\theta^2),
\]
and in general the right-hand side changes by
$2\sigma(R_1+R_2)+\sigma^2$. We use the normalized null coordinates
\begin{equation}
 u=\frac{\tau+\xi_1}{2\lambda},\qquad
 v=\frac{\tau-\xi_1}{2\lambda\theta^2},\qquad
 z=\frac{\xi_2}{\lambda\theta},
 \label{eq:planar-lorentz-rescaling}
\end{equation}
and write $L$ for the linear map $(\xi,\tau)\mapsto(u,v,z)$. Since
$4uv-z^2=(\tau^2-|\xi|^2)/(\lambda^2\theta^2)$, $L$ maps the light cone
onto $\{4uv=z^2\}$. By the choice of $b_*$ we have $u\simeq1$; moreover
$|v|+|z|\lesssim1$, and
\begin{align*}
 4uv-z^2
 &=\frac{2R_1R_2(1-\cos(\alpha-\beta))
                +2\sigma(R_1+R_2)+\sigma^2}{\lambda^2\theta^2}
 \\
 &=O\left(\rho+\frac{1}{\lambda\theta^2}
                  +\frac{1}{\lambda^2\theta^2}\right)
 =O(\rho),
\end{align*}
where we used $\mu\theta^2\gtrsim1$. Since $u$ is bounded
away from zero, the support lies in the $O(\rho)$ neighborhood of the
cone $v=z^2/(4u)$.

We next determine the position of the normalized Fourier support
of each $F_\nu$. Set
\[
 u_*=\frac{R_1\cos^2(\alpha/2)}{\lambda},\qquad
 \varphi=\frac{\tan(\alpha/2)}{\theta},\qquad
 \varphi_\nu=\frac{\tan(\alpha_\nu/2)}{\theta}.
\]
The contribution of the high-frequency wave to
$(u,v,z)$ is $u_*(1,\varphi^2,2\varphi)$,
with $u_*\simeq1$ and $|\varphi|\lesssim1$. The contribution of the
low-frequency wave is $O(\rho)$ in each coordinate:
\[
 \frac{R_2(1+\cos\beta)}{2\lambda}=O(\rho),\qquad
 \frac{R_2(1-\cos\beta)}{2\lambda\theta^2}=O(\rho),\qquad
 \frac{R_2\sin\beta}{\lambda\theta}=O(\rho),
\]
and the contribution of the time weight,
$\bigl(\sigma/(2\lambda),\sigma/(2\lambda\theta^2),0\bigr)$, is also
$O(\rho)$. Hence every point of the normalized Fourier support of
$F_\nu$ has the form
\begin{equation}
 (u,v,z)=u_*(1,\varphi^2,2\varphi)+E,\qquad
 |E|\lesssim\rho,\qquad |\varphi-\varphi_\nu|\lesssim\rho^{1/2}.
 \label{eq:planar-fine-support-parametrization}
\end{equation}
In particular, $z/(2u)=\varphi+O(\rho)$; the
radial frequencies affect only $u_*$, and the low-frequency wave
changes the direction only by $O(\rho)$.

In the coordinates
\[
 T_\nu=z-2\varphi_\nu u,\qquad
 N_\nu=v-\varphi_\nu z+\varphi_\nu^2u,
\]
\eqref{eq:planar-fine-support-parametrization} gives
\[
 T_\nu=2u_*(\varphi-\varphi_\nu)+O(\rho)=O(\rho^{1/2}),\qquad
 N_\nu=u_*(\varphi-\varphi_\nu)^2+O(\rho)=O(\rho).
\]
Thus the normalized Fourier support of $F_\nu$ lies in the sector
\[
 \Pi_\nu
 =\left\{c\leq u\leq C,\quad
    |z-2\varphi_\nu u|\leq C\rho^{1/2},\quad
    |v-\varphi_\nu z+\varphi_\nu^2u|\leq C\rho\right\}
\]
for fixed constants $c,C$. The map
$(u,v,z)\mapsto(u,T_\nu,N_\nu)$
and its inverse have uniformly bounded coefficients and determinant of
modulus one. Hence, up to fixed changes of the dimensions and a fixed
number of radial pieces, the sets $\Pi_\nu$ are the standard sectors
of dimensions $1\times\rho^{1/2}\times\rho$ of the $O(\rho)$
neighborhood of the cone.

The parameters $\varphi_\nu$ of neighboring subcaps are separated by
distances comparable to $\rho^{1/2}$, and
$|z/(2u)-\varphi_\nu|\lesssim\rho^{1/2}$ on $\Pi_\nu$.
Consequently each $\Pi_\nu$ meets a bounded number of standard sectors
of thickness $\rho$, and each standard sector meets the normalized
Fourier supports of a bounded number of the $F_\nu$, uniformly in
$\lambda,\mu,\theta$.

Let $H_\nu$ be the function with $\widehat{H_\nu}=\widehat{F_\nu}\circ L^{-1}$.
Since $L$ is linear, $\|H_\nu\|_{L^r}=c_L\|F_\nu\|_{L^r}$ with the same
constant $c_L$ for all $\nu$, and likewise for $\sum_\nu H_\nu$.
Identifying $\{4uv=z^2\}$ with the standard light cone and making a
fixed finite radial subdivision, we see that $\sum_\nu H_\nu$ has
Fourier support in
a fixed truncated cone neighborhood of thickness $C\rho$. If $\rho$ is
bounded below, the number of subcaps is bounded and
\eqref{eq:planar-output-decoupling} follows from the triangle
inequality; hence we may assume that $C\rho$ is small.

We use the smooth formulation of the decoupling inequality, with a
partition of unity $\{Q_k\}$ subordinate to sectors of this
neighborhood. The multipliers $Q_k$ have kernels uniformly bounded
in $L^1$. Explicitly, let
$\{\chi_k\}$ be a smooth partition of unity in the variable
$z/(2u)$ at scale $\rho^{1/2}$, with centers
$\varphi_k$, and let
\[
 m_k(u,v,z)
 =\chi_0(u)\chi_k\!\left(\frac{z}{2u}\right)
   \chi_*\!\left(\frac{v-\varphi_k z+\varphi_k^2u}{\rho}\right),
\]
where $\chi_0$ equals one on the relevant interval of $u$
and $\chi_*$ equals one on a sufficiently large fixed interval. On the
support of $\chi_k$, the identity
\[
 v-\varphi_k z+\varphi_k^2u
 =v-\frac{z^2}{4u}
   +u\left(\frac{z}{2u}-\varphi_k\right)^2
\]
shows that $\sum_km_k=1$ on the Fourier support of
$\sum_\nu H_\nu$. In the affine coordinates
\[
 u,\qquad \frac{z-2\varphi_k u}{\rho^{1/2}},\qquad
 \frac{v-\varphi_k z+\varphi_k^2u}{\rho},
\]
the symbols $m_k$ are smooth and compactly supported with uniform
bounds, since $u$ is bounded away from zero; hence their
kernels are uniformly bounded in $L^1$.

Define
\[
 \mathfrak n(k)
 =\{\nu:\operatorname{supp}\widehat H_\nu
          \cap\operatorname{supp}m_k\ne\varnothing\}.
\]
By the two incidence bounds and the kernel bounds, for
$2\leq r\leq6$,
\begin{equation}\label{eq:planar-plate-incidence-norm}
 \begin{aligned}
 \Big(\sum_k\Big\|Q_k\sum_\nu H_\nu\Big\|_r^2\Big)^{1/2}
 &\le\Big(\sum_k\Big(\sum_{\nu\in\mathfrak n(k)}
     \|Q_kH_\nu\|_r\Big)^2\Big)^{1/2}\\
 &\lesssim\Big(\sum_\nu\|H_\nu\|_r^2\Big)^{1/2}.
 \end{aligned}
\end{equation}
The $\ell^2$ decoupling theorem for the cone in $\R^3$ of
Bourgain--Demeter \cite[Theorem~1.2]{BD2015} holds for $2\leq r\leq6$. Applying
it at thickness $C\rho$ together with
\eqref{eq:planar-plate-incidence-norm}, and undoing the change of
variables, we obtain \eqref{eq:planar-output-decoupling}; the constant
$c_L$ appears on both sides and cancels.
\end{proof}

\subsection{Estimates at a fixed angular scale}
Combining \cref{lem:cone-restriction}, the kernel estimates
\eqref{eq:planar-cap-kernels}, and \cref{lem:product-decoupling}, we
obtain estimates for $\mathfrak t_\theta$ in a range of exponents.

\begin{lemma}\label{lem:planar-angular-scale}
Let $5/3\le s\le2$ and $\eps>0$, and let $p_1,p_2,r$ satisfy
\begin{enumerate}
\item[(i)] $p_1,p_2\ge2$ and $2\le r\le6$;
\item[(ii)] $1/p_1+1/p_2\ge1-1/r$;
\item[(iii)] the numbers $s/r$, $1/p_1-s/(2r)$, $1/p_2-s/(2r)$, and
$1-1/p_1-1/p_2$ are nonnegative.
\end{enumerate}
Let
\begin{equation}\label{eq:planar-a-gamma}
 A=\frac1{2p_1}+\frac1{2p_2}-\frac1r,\qquad
 \gamma_\eps=\frac14-\frac1{2p_1}+\frac1{2p_2}
                 +\Big(\frac{5s}{4}-3-s\eps\Big)\frac1r.
\end{equation}
Then, with $\mathfrak t_\theta$ as in \eqref{eq:planar-bt}, uniformly
for $\mu^{-1/2}\lesssim\theta\lesssim1$,
\begin{equation}\label{eq:planar-angular-scale}
 \|\mathfrak t_\theta(f,g)\|_{L^r}
 \lesssim_{s,\eps}
 (\lambda\mu)^{A}
 \rho^{\gamma_\eps-\eps}
 \theta^{2-1/p_1-1/p_2-3/r}
 \|f\|_{L^{p_1}}\|g\|_{L^{p_2}}.
\end{equation}
\end{lemma}

\begin{proof}
Throughout the proof we write $x_i=1/p_i$, $i=1,2$, and $y=1/r$.

\emph{Interpolation on subcaps.}
Fix a pair $\omega\sim\omega'$ and divide $\omega$ into subcaps
$\omega_\nu$ of width
\begin{equation}
 \delta=\rho^{1/2}\theta.
 \label{eq:planar-h}
\end{equation}
There are $O(\rho^{-1/2})$ subcaps, and $\delta\gtrsim\lambda^{-1/2}$,
so \eqref{eq:planar-cap-kernels} applies to them. Let
$f_\nu=P_{\omega_\nu}f_{\omega}$, $g_{\omega'}=P_{\omega'}g$, and
\[
 F_\nu(x,t)=V_\lambda f_\nu(x,t)\,V_\mu g_{\omega'}(x,t).
\]
Since $w$ is bounded, \cref{lem:cone-restriction} and
\eqref{eq:planar-cap-kernels} give the following four estimates for the
bilinear operator $(f,g)\mapsto F_\nu$:
\begin{equation}
\addtolength{\arraycolsep}{10pt}
\begin{array}{c|c}
\text{spaces}&\text{bound}\\[3pt]
L^2\times L^2\longrightarrow L^s_{x,t}
 &\lambda^{1/s-1}\mu^{2-4/s}\theta^{1-3/s}\rho^{-\eps}\\[3pt]
L^1\times L^\infty\longrightarrow L^\infty_{x,t}
 &\lambda\theta\\[3pt]
L^\infty\times L^1\longrightarrow L^\infty_{x,t}
 &\mu \delta\\[3pt]
L^\infty\times L^\infty\longrightarrow L^\infty_{x,t}
 &\delta\theta.
\end{array}
\label{eq:planar-four-endpoints}
\end{equation}
To interpolate, we insert slightly enlarged cutoffs that equal one on
the Fourier supports of $f_\nu$ and $g_{\omega'}$. The enlargements are
fixed and preserve the separated-cap hypotheses. The resulting
bilinear operator acts on arbitrary inputs, so its interpolated
estimate applies to $f_\nu$ and $g_{\omega'}$ themselves.
We interpolate with the weights
\begin{equation}
 sy,\qquad x_1-\frac{sy}{2},\qquad
 x_2-\frac{sy}{2},\qquad 1-x_1-x_2,
 \label{eq:planar-weights}
\end{equation}
which are nonnegative by (iii), sum to one, and give the input
exponents $x_1,x_2$ and the output exponent $y$. This gives
\begin{equation}
 \begin{aligned}
 \|F_\nu\|_r
 &\lesssim
 \lambda^{x_1+y-3sy/2+sy\eps}
 \mu^{x_2+3sy/2-4y-sy\eps}\\
 &\qquad{}\times \delta^{1-x_1-sy/2}
 \theta^{1-x_2+sy/2-3y}
 \|f_\nu\|_{p_1}\|g_{\omega'}\|_{p_2}.
 \end{aligned}
 \label{eq:planar-fine-interpolation}
\end{equation}
For instance, the power of $\lambda$ is
$(1/s-1+\eps)sy+(x_1-sy/2)$.

\emph{Decoupling.}
By \eqref{eq:analysis}, since $p_1\ge2$,
$(\sum_\nu\|f_\nu\|_{p_1}^{p_1})^{1/p_1}\lesssim\|f_{\omega}\|_{p_1}$,
and since there are $O(\rho^{-1/2})$ subcaps, H\"older's inequality
gives
\begin{equation}
 \Big(\sum_\nu\|f_\nu\|_{p_1}^2\Big)^{1/2}
 \lesssim\rho^{-1/4+x_1/2}\|f_{\omega}\|_{p_1}.
 \label{eq:planar-fine-input-sum}
\end{equation}
We substitute \eqref{eq:planar-h} into
\eqref{eq:planar-fine-interpolation} and combine the result with
\eqref{eq:planar-output-decoupling} and
\eqref{eq:planar-fine-input-sum}. The substitution gives the factor
$\rho^{(1-x_1-sy/2)/2}$, and \eqref{eq:planar-fine-input-sum} gives
the factor $\rho^{-1/4+x_1/2}$. With $A$ and $\gamma_\eps$ as in
\eqref{eq:planar-a-gamma}, this gives
\begin{equation}
 \|V_\lambda f_{\omega}\,V_\mu g_{\omega'}\|_{L^r_{x,t}}
 \lesssim_\eps
 \lambda^{-y}(\lambda\mu)^A
 \rho^{\gamma_\eps-\eps}
 \theta^{2-x_1-x_2-3y}
 \|f_{\omega}\|_{p_1}\|g_{\omega'}\|_{p_2}.
 \label{eq:planar-coarse-spacetime}
\end{equation}

\emph{Summation over directions.}
The spatial Fourier support of $V_\lambda f_\omega V_\mu g_{\omega'}$ is
contained in a rectangle of length $O(\lambda)$ and width
$O(\lambda\theta)$ in the direction of $\omega$. After a fixed finite
partition in $|\xi|/\lambda$ into intervals of small fixed length,
whose multipliers have kernels uniformly bounded in $L^1$, these
rectangles and their fixed dilates stay away from the origin, and for
caps in a small fixed aperture each of them meets a bounded number of
the rectangles associated with other caps at scale $\theta$; the
finitely many caps of large aperture are treated separately. Since
$2\le r<\infty$, \cref{lem:rectangles} gives
\[
 \|\mathfrak b_\theta(f,g)\|_{L^r_{x,t}}
 \lesssim
 \Big(\sum_{\omega\sim \omega'}
   \|V_\lambda f_{\omega}\,V_\mu g_{\omega'}\|_{L^r_{x,t}}^{r'}\Big)^{1/r'}.
\]
Let $a_{\omega}=\|f_{\omega}\|_{p_1}$ and $b_{\omega'}=\|g_{\omega'}\|_{p_2}$.
By (ii), H\"older's inequality, the bounded number of $\omega'$ with
$\omega\sim\omega'$, and the inclusion $\ell^{u}\subset\ell^{r'}$ for
$1/u=x_1+x_2\ge1/r'$,
\begin{equation}
 \|(a_{\omega}b_{\omega'})_{\omega\sim \omega'}\|_{\ell^{r'}}
 \lesssim\|a\|_{\ell^{p_1}}\|b\|_{\ell^{p_2}}
 \lesssim\|f\|_{p_1}\|g\|_{p_2},
 \label{eq:planar-coarse-input-condition}
\end{equation}
where the last inequality follows from \eqref{eq:analysis}, since
$p_1,p_2\ge2$.

\emph{The supremum in time.}
The Fourier transform of $\mathfrak b_\theta(f,g)$ in $t$ is supported
in an interval of length $O(\lambda)$, so
\cref{lem:modulated-bernstein} gives
\[
 \|\mathfrak t_\theta(f,g)\|_r
 \lesssim\lambda^{1/r}\|\mathfrak b_\theta(f,g)\|_{L^r_{x,t}}.
\]
The factor $\lambda^{1/r}=\lambda^y$ cancels $\lambda^{-y}$ in
\eqref{eq:planar-coarse-spacetime}. Combining this with
\eqref{eq:planar-coarse-spacetime}--\eqref{eq:planar-coarse-input-condition}
proves \eqref{eq:planar-angular-scale}.
\end{proof}

\subsection{Summation over angular scales}\label{sec:planar-sum}
We now sum the estimates of \cref{lem:planar-angular-scale} over the
angular scales $\theta$ for $(1/p_1,1/p_2,1/r)$ near $(2/5,2/5,2/5)$,
first in weak type and then in strong type by multilinear
interpolation. We fix $s\in(19/10,2)$, for instance $s=39/20$. At
$1/p_1=1/p_2=1/r=2/5$ the weights \eqref{eq:planar-weights} are $2s/5$,
$(2-s)/5$, $(2-s)/5$, and $1/5$, which are positive; conditions (i) and
(ii) of \cref{lem:planar-angular-scale} hold with strict inequality;
and $A=0$ and $\gamma_\eps=s/2-19/20-2s\eps/5$. Since $s>19/10$, we may
choose $0<\eps\le1/40$ and a neighborhood $\mathcal B$ of $(2/5,2/5,2/5)$
so small that the hypotheses of \cref{lem:planar-angular-scale} hold
for $(1/p_1,1/p_2,1/r)\in\mathcal B$ and that
$\gamma_\eps-\eps\ge c:=(s/2-19/20)/2$ there. Hence we may replace
$\rho^{\gamma_\eps-\eps}$ by $\rho^c$ in \eqref{eq:planar-angular-scale}
on $\mathcal B$.

For fixed $p_1,p_2$ define $r_0$ by $1/r_0=(2-1/p_1-1/p_2)/3$, and
choose $\zeta>0$ so small that $r_\pm$, given by $1/r_\pm=1/r_0\pm\zeta$,
lie in $\mathcal B$. For $\theta=2^{-m}$, the powers of $\theta$ in
\eqref{eq:planar-angular-scale} with $r=r_\pm$ are $2^{\pm3\zeta m}$. We show that the bisublinear operator
\[
 \mathfrak s(f,g)=\sum_{\theta}\mathfrak t_\theta(f,g),
\]
where $\mathfrak t_\theta$ is given by \eqref{eq:planar-bt} and the sum
is over the dyadic angular scales $\mu^{-1/2}\lesssim\theta\lesssim1$,
satisfies
\begin{equation}
 \|\mathfrak s(f,g)\|_{L^{r_0,\infty}}
 \lesssim
 (\lambda\mu)^{\frac1{2p_1}+\frac1{2p_2}-\frac1{r_0}}\rho^c
 \|f\|_{L^{p_1}}\|g\|_{L^{p_2}}.
 \label{eq:planar-angular-weak}
\end{equation}
With the input exponents fixed, Bourgain's summation argument applied
to the estimates with $r=r_\pm$ gives
$\|\mathfrak s(f,g)\|_{L^{r_0,\infty}}\lesssim(M_+M_-)^{1/2}$;
see \cite{Lee2003} for details. Here
\[
 M_\pm=(\lambda\mu)^{\frac1{2p_1}+\frac1{2p_2}-\frac1{r_\pm}}
           \rho^c\|f\|_{p_1}\|g\|_{p_2}.
\]
Since $1/r_0=(1/r_++1/r_-)/2$, this is \eqref{eq:planar-angular-weak}.

We now obtain a strong type estimate by multilinear real
interpolation. Let $a_0=2/5$, and let $h>0$ be small. Consider
the three points
\begin{equation}
 (a_0+h,a_0),\qquad(a_0,a_0+h),\qquad
 (a_0-h,a_0-h)
 \label{eq:planar-interpolation-triangle}
\end{equation}
in the $(1/p_1,1/p_2)$-plane, with output exponents given by
$1/r=(2-1/p_1-1/p_2)/3$. For $h$ small, \eqref{eq:planar-angular-weak}
holds at these points, and their barycenter is $(2/5,2/5)$, where
$1/r=2/5$. In the notation of \cite[Corollary~1.1]{GLLZ2012}, the
relevant determinants are, up to sign,
\[
 \gamma_0=3h^2,\qquad \gamma_1=\gamma_2=h^2;
\]
indeed, the column of output exponents is $2/3$ times the
constant column minus $1/3$ times each input column.
Replacing either input column by this column therefore
multiplies the determinant by $-1/3$. Since the determinants do not
vanish and $1/r\leq1/p_1+1/p_2$, \cite[Corollary~1.1]{GLLZ2012}
applies to
$\mathfrak s$ and gives a strong type bound at the barycenter, with
constant bounded by the geometric mean of the three weak type
constants. The powers of $\lambda\mu$ at the three points are
$5h/6$, $5h/6$, and $-5h/3$, whose mean is zero, while
the factor $\rho^c$ is common to all three. Hence
\begin{equation}
 \|\mathfrak s(f,g)\|_{5/2}
 \lesssim\rho^c\|f\|_{5/2}\|g\|_{5/2}.
 \label{eq:planar-transverse-strong}
\end{equation}

\begin{figure}[htbp]
\centering
\begin{tikzpicture}[x=2.1cm,y=1.7cm,>=Stealth,font=\small]
  \coordinate (a) at (1,0);
  \coordinate (b) at (0,1);
  \coordinate (c) at (-1,-1);
  \coordinate (t) at (0,0);
  \fill[blue!7] (a)--(b)--(c)--cycle;
  \draw[->,gray!65] (-1.4,0)--(1.5,0) node[right] {$1/p_1$};
  \draw[->,gray!65] (0,-1.42)--(0,1.5) node[above] {$1/p_2$};
  \draw[blue!65!black,thick] (a)--(b)--(c)--cycle;
  \draw[densely dotted,gray] (a)--($(b)!.5!(c)$);
  \draw[densely dotted,gray] (b)--($(a)!.5!(c)$);
  \draw[densely dotted,gray] (c)--($(a)!.5!(b)$);
  \foreach \p in {a,b,c} \fill[blue!65!black] (\p) circle[radius=2pt];
  \fill[orange!85!black] (t) circle[radius=2.8pt];
  \node[above right,align=left] at (a) {$v_1=(2/5+h,2/5)$};
  \node[above left,align=right] at (b) {$v_2=(2/5,2/5+h)$};
  \node[below left,align=right] at (c) {$v_3=(2/5-h,2/5-h)$};
  \node[below right,fill=white,inner sep=2pt] at ($(t)+(.04,-.03)$)
       {$v_*=(2/5,2/5)$};
\end{tikzpicture}
\caption{The interpolation points \eqref{eq:planar-interpolation-triangle}
in the $(1/p_1,1/p_2)$-plane, $0<h\ll1$. At each vertex
\eqref{eq:planar-angular-weak} holds with $1/r=(2-1/p_1-1/p_2)/3$. At
the barycenter $v_*$ we have $1/r=2/5$, and multilinear interpolation gives
\eqref{eq:planar-transverse-strong}. The axes are drawn through
$(2/5,2/5)$.}
\label{fig:planar-simplex}
\end{figure}
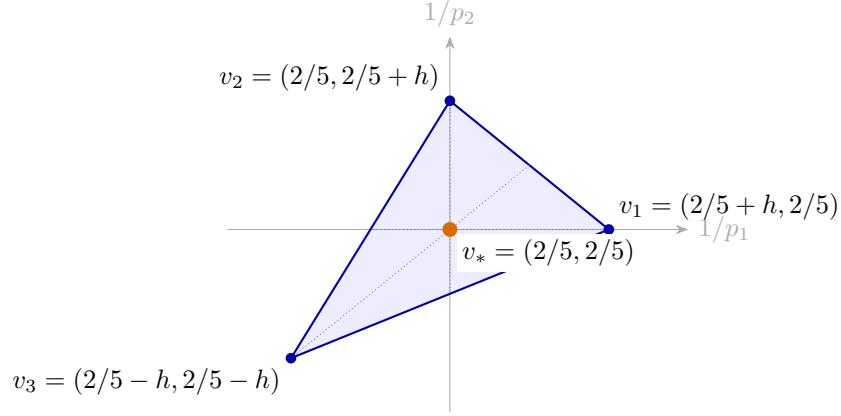

\subsection{Near-diagonal and comparable frequencies}
\label{sec:planar-completion}
It remains to treat the near-diagonal pairs of caps and the case of
comparable frequencies.

\emph{The near-diagonal part.}
The remaining pairs of directions are at distance at most
$\theta_0\simeq\mu^{-1/2}$. We cover them by pairs of caps of width
$\theta_0$ with bounded multiplicity, and we divide the high-frequency
caps into subcaps of width
$\delta=\rho^{1/2}\theta_0\simeq\lambda^{-1/2}$. For each $t$,
$\|V_\lambda f_\nu(t)\|_2\le|w(t)|\lambda^{-1/2}\|f_\nu\|_2$, and since
the low-frequency sector has area $O(\mu^2\theta_0)=O(\mu^{3/2})$,
Bernstein's inequality gives
\[
 \|V_\mu g_{\omega'}(t)\|_\infty
 \lesssim |w(t)|\mu^{1/4}\|g_{\omega'}\|_2.
\]
Integrating $|w(t)|^4$ in $t$, we obtain
\begin{equation}
 \|V_\lambda f_\nu\, V_\mu g_{\omega'}\|_{L^2_{x,t}}
 \lesssim\lambda^{-1/2}\mu^{1/4}
                 \|f_\nu\|_2\|g_{\omega'}\|_2.
 \label{eq:planar-terminal-l2}
\end{equation}
Interpolation between \eqref{eq:planar-terminal-l2}, with weight $4/5$,
and the $L^\infty\times L^\infty$ estimate in
\eqref{eq:planar-four-endpoints} gives
\[
 \|V_\lambda f_\nu\, V_\mu g_{\omega'}\|_{L^{5/2}_{x,t}}
 \lesssim\lambda^{-2/5}\rho^{1/10}
                 \|f_\nu\|_{5/2}\|g_{\omega'}\|_{5/2}.
\]
\Cref{lem:product-decoupling} with $\theta=\theta_0$, which requires no
separation of the caps, and \eqref{eq:planar-fine-input-sum} with
$p_1=5/2$ cost a factor $\rho^{-\eps-1/20}$. The sum over the caps and
the supremum in $t$ are handled as in the proof of
\cref{lem:planar-angular-scale}, and Bernstein's inequality in $t$
cancels the factor $\lambda^{-2/5}$. Thus the near-diagonal part
satisfies \eqref{eq:planar-model-gap} with the factor
$\rho^{1/20-\eps}\le\rho^c$. Together with
\eqref{eq:planar-transverse-strong}, this proves
\eqref{eq:planar-model-gap} for $\lambda\ge16\mu$.

\emph{Comparable frequencies.}
If $\lambda\simeq\mu$, we use \cite[Proposition~1.2]{Lee2003}; its
condition $q>14/3$ holds at $(p,q)=(5/2,5)$. Covering $[1/2,3]$ by the
fixed dilates $[1/2,1]$, $[1,2]$, and $[2,4]$ of $[1,2]$, and applying
the corresponding space-time dilations, we obtain
\[
 \|U(t)P_\ell f\|_{L^5(\R^2\times[1/2,3])}
 \lesssim \ell^{3/10}\|f\|_{5/2}.
\]
Set $J=[1/2,3]$. The fundamental theorem of calculus and
H\"older's inequality give
\[
 \left\|\sup_{1\le t\le2}|F(\cdot,t)|\right\|_{L^5_x}^{5}
 \lesssim
 \|F\|_{L^5(\R^2\times J)}^{5}
 +\|F\|_{L^5(\R^2\times J)}^{4}
  \|\partial_tF\|_{L^5(\R^2\times J)}.
\]
Applying this to $F(x,t)=U(t)P_\ell f(x)$, and using the
preceding estimate and its differentiated version with an
additional factor $\ell$, we obtain
\[
 \left\|\sup_{1\leq t\leq2}|U(t)P_\ell f|\right\|_5
 \lesssim \ell^{1/5}\ell^{3/10}\|f\|_{5/2}
 =\ell^{1/2}\|f\|_{5/2}.
\]
Hence $\|\sup_{1\le t\le2}|V_\ell f|\|_5\lesssim\|f\|_{5/2}$, and
H\"older's inequality gives \eqref{eq:planar-model-gap} with a constant
in place of $\rho^c$. This also gives the estimate for a single
frequency used in \cref{sec:spherical}.

\subsection{The planar \texorpdfstring{$L^2$}{L2} estimate}
\label{sec:planar-l2}
Let $0<\gamma_2<1/12$. Since $(4-p)/[24(p-1)]\to1/12$ as $p\to2$, we
can choose $p\in(2,5/2)$ with $(4-p)/[24(p-1)]>\gamma_2$. Define $r$ by
$1/r=(2-2/p)/3$, and set $\kappa=p/[2(p-1)]$ and $b=\gamma_2/\kappa$, so
that $b<(4-p)/(12p)$. We first show that
\begin{equation}\label{eq:planar-x-gap}
 \|\sup_{1\le t\le2}|V_\lambda f\,V_\mu g|\|_{r}
 \lesssim(\lambda\mu)^{1/p-1/r}\rho^b\|f\|_{p}\|g\|_{p}.
\end{equation}

We apply \cref{lem:planar-angular-scale} with $s=2$ near
$(1/p_1,1/p_2,1/r)=(1/p,1/p,1/r)$. At this point the weights
\eqref{eq:planar-weights} are $2/r$, $1/p-1/r$, $1/p-1/r$, and $1-2/p$,
which are positive; moreover $p>2$, $5/2<r<3$, and $2/p>1-1/r$. Hence
the hypotheses of \cref{lem:planar-angular-scale} hold in a
neighborhood of this point. At the point we have $A=1/p-1/r$ and
$\gamma_\eps=(4-p)/(12p)-2\eps/r$, so we may choose $\eps>0$ and the
neighborhood so small that $\gamma_\eps-\eps>b$ on it. The argument of
\cref{sec:planar-sum}, with $a_0=1/p$ in
\eqref{eq:planar-interpolation-triangle} and with output exponents
given by $1/r=(2-1/p_1-1/p_2)/3$, gives \eqref{eq:planar-x-gap} for the
sum over the angular scales; the powers of $\lambda\mu$ at the three
points average to $1/p-1/r$.

For the near-diagonal part we interpolate \eqref{eq:planar-terminal-l2}
and the other three estimates in \eqref{eq:planar-four-endpoints}, with
$\theta=\theta_0$ and $\delta\simeq\lambda^{-1/2}$, using the weights
$2/r$, $1/p-1/r$, $1/p-1/r$, $1-2/p$. The four bounds are
$\lambda^{-1/2}\mu^{1/4}$, $\lambda\mu^{-1/2}$, $\mu\lambda^{-1/2}$,
and $\lambda^{-1/2}\mu^{-1/2}$, so the estimate for a subcap and a cap
has the factor $\lambda^{(3/p-3/r-1)/2}\mu^{(3/p-1)/2}$. As in
\cref{sec:planar-completion}, the sum over the subcaps and
\cref{lem:product-decoupling} cost $\rho^{-1/4+1/(2p)-\eps}$, and
Bernstein's inequality in $t$ contributes $\lambda^{1/r}$. Since
$2/p+3/r=2$, the result is
$(\lambda\mu)^{1/p-1/r}\rho^{(4-p)/(12p)-\eps}$, which proves
\eqref{eq:planar-x-gap} for $\eps$ small.

By \eqref{eq:planar-cap-kernels} with full aperture,
\[
 \|\sup_{1\le t\le2}|V_\lambda f\,V_\mu g|\|_\infty
 \lesssim\lambda\mu\|f\|_1\|g\|_1.
\]
We interpolate this with \eqref{eq:planar-x-gap}, with weight $\kappa$
on the latter. Since $\kappa/p+1-\kappa=1/2$, $\kappa/r=1/3$, and
$\kappa(1/p-1/r)+1-\kappa=1/6$, we obtain
\[
 \|\sup_{1\le t\le2}|V_\lambda f\,V_\mu g|\|_3
 \lesssim(\lambda\mu)^{1/6}\rho^{\gamma_2}\|f\|_2\|g\|_2.
\]
Apply this estimate to $f_\lambda,g_\mu$, using cutoffs equal to one
on their Fourier supports. Since $w^2$ is bounded below on $[1,2]$,
removing $w^2$ and the factor $\lambda^{-1/2}\mu^{-1/2}$ gives
\[
 \|\sup_{1\le t\le2}|Uf_\lambda\, Ug_\mu|\|_3
 \lesssim(\lambda\mu)^{2/3}\rho^{\gamma_2}\|f_\lambda\|_2\|g_\mu\|_2.
\]
For $\lambda\simeq\mu$ we interpolate the estimate for comparable
frequencies in \cref{sec:planar-completion} with the
$L^1\times L^1\to L^\infty$ estimate, with weight $5/6$ on the former.
This proves \cref{prop:wave-gap} for $d=2$.

\section{Proof of the wave maximal theorem}\label{sec:wave}
Fix $0<\gamma_d<\gamma_d^*$. We first extend \eqref{eq:wave-gap-local}
to all times and all frequencies, and then sum over dyadic
frequencies. Replacing $f_\lambda$ and
$g_\mu$ in \eqref{eq:wave-gap-local} by $U(-1)f_\lambda$ and
$U(-1)g_\mu$, which have the same Fourier supports and $L^2$
norms, we obtain the same estimate with the supremum over $[0,1]$.

Let $\lambda\ge\mu>0$, and let $T>0$ be large, so that $T\mu\ge2$. The
unitary dilation $D_Tf(x)=T^{d/2}f(Tx)$ maps functions of frequency
$\lambda$ to functions of frequency $T\lambda$, and
$U(t)D_Tf(x)=T^{d/2}(U(Tt)f)(Tx)$. Applying the estimate on $[0,1]$ to
$D_Tf_\lambda$ and $D_Tg_\mu$, and noting that $\rho$ is invariant under
this scaling, we obtain
\[
 T^{d-d/r_d}\|\sup_{0\le t\le T}|Uf_\lambda\,Ug_\mu|\|_{r_d}
 \lesssim T^{2s_c}(\lambda\mu)^{s_c}\rho^{\gamma_d}
 \|f_\lambda\|_2\|g_\mu\|_2.
\]
Since $d-d/r_d=2s_c$, the powers of $T$ cancel. Letting $T\to\infty$,
and using \cref{rem:signs} for negative times, since
$U(-t)=U_-(t)$, we obtain, for all
$\lambda\ge\mu>0$,
\begin{equation}\label{eq:wave-gap-global}
 \|\sup_{t\in\R}|Uf_\lambda\,Ug_\mu|\|_{r_d}
 \lesssim(\lambda\mu)^{s_c}\rho^{\gamma_d}\|f_\lambda\|_2\|g_\mu\|_2.
\end{equation}

Let $f_j=P_{2^j}f$, $j\in\Z$, and $b_j=2^{js_c}\|f_j\|_2$. For a finite
sum, we expand the square, use the triangle inequality in $L^{r_d}$,
and apply \eqref{eq:wave-gap-global} to each pair $(f_j,f_k)$, for which
$\rho\simeq2^{-|j-k|}$. This gives
\[
 \Big\|\sup_t\Big|U(t)\sum_jf_j\Big|\Big\|_{2r_d}^2
 \lesssim\sum_{j,k}2^{-\gamma_d|j-k|}b_jb_k
 \lesssim\sum_jb_j^2
 \lesssim\|f\|_{\dot H^{s_c}}^2,
\]
where the second inequality is Young's inequality on $\ell^2(\Z)$.
Since $2r_d=q_c$, this proves \eqref{eq:wave-main} for finite sums.
Applying the same bound to the tails gives convergence in the norm
$\|\sup_{t\in\mathbb Q}|\cdot|\|_{q_c}$. A subsequence converges
uniformly in $t$ for almost every $x$: for differences of finite
sums, continuity makes the supremum over $\mathbb Q$ equal to that
over $\R$. The resulting continuous-in-time representative agrees
with $U(t)f$ at each fixed $t$, by the embedding
$\dot H^{s_c}\hookrightarrow L^{q_c}$ and the unitarity of $U(t)$ on
$\dot H^{s_c}$. This proves \cref{thm:wave}.
\section{Proof of the spherical maximal theorem}\label{sec:spherical}
We prove \cref{thm:main} using \eqref{eq:wave-gap-local} when $d\ge3$
and \eqref{eq:planar-model-gap} when $d=2$. In both cases we obtain a
bilinear estimate for spherical averages, and then sum over
frequencies by \cref{lem:power-sum}.

\subsection{Reduction to half-wave propagators}
We express the frequency localized spherical averages in terms of the
half-wave propagators $U_\pm(t)$, with amplitudes that can be
removed by a Fourier expansion in $t$.
For $\ell\ge2$ and $1\le t\le2$ we write
\begin{equation}\label{eq:bessel}
 \widehat\sigma(t\xi)P_\ell(\xi)
 =\ell^{-(d-1)/2}\sum_{\pm}
 e^{\pm it|\xi|}a_{\pm,\ell}(t,\xi)P_\ell(\xi).
\end{equation}
This exact decomposition is obtained by writing the Bessel function
as the sum of its two Hankel components. Their asymptotic expansions give
$|\partial_t^k\partial_\xi^\alpha a_{\pm,\ell}|\lesssim_{k,\alpha}\ell^{-|\alpha|}$
on the support of $P_\ell$. We multiply each amplitude by a smooth
annular cutoff equal to one there, extend it smoothly in $t$
to a function supported in a slightly larger interval of length less
than $T_0$, and expand it in a Fourier series in $t$ with period
$T_0$:
\[
 a_{\pm,\ell}(t,\xi)=\sum_{m\in\Z}e^{2\pi imt/T_0}a_{\pm,\ell,m}(\xi).
\]
Integration by parts in $t$ and kernel estimates for the rescaled
symbols give, for every $N$,
\begin{equation}\label{eq:amplitude}
 \|a_{\pm,\ell,m}(D)\|_{L^p\to L^p}
 \lesssim_{N,p}(1+|m|)^{-N},\qquad1\le p\le\infty.
\end{equation}
When $d=2$ we expand $a_{\pm,\ell}/w$ instead, using an interval on
which $w$ does not vanish. In a product of two terms of these
expansions, the factor $e^{2\pi i(m+n)t/T_0}$ has modulus one and can
be removed before taking the supremum in $t$. By the triangle
inequality and \cref{rem:signs}, it suffices to consider products of
two terms with the sign $+$; note that
$\overline{a_{\pm,\ell,m}(D)g}=\widetilde a(D)\bar g$ with
$\widetilde a(\xi)=\overline{a_{\pm,\ell,m}(-\xi)}$, which satisfies
\eqref{eq:amplitude} as well. Hence the bilinear estimates of the
previous sections apply to each product with constants independent of
$m,n$, and the rapid decay in \eqref{eq:amplitude} allows us to sum
over $m,n$.

\subsection{A bilinear estimate for spherical averages}
We now prove \eqref{intro:eq:spherical-gap}, first for $d\ge3$ and then
for $d=2$.
Let $d\ge3$, fix $\gamma_d$ as in \cref{prop:wave-gap}, and set
\[
 \alpha_d=\frac{d-1}{2}-s_c>0,\qquad
 \vartheta_d=\frac1{1+\alpha_d}=\frac{2(d+1)}{d^2+1}.
\]
By \eqref{eq:bessel}, \eqref{eq:amplitude}, and
\eqref{eq:wave-gap-local},
\[
 \|\sup_{1\le t\le2}|A_tP_\lambda f\, A_tP_\mu g|\|_{r_d}
 \lesssim(\lambda\mu)^{-\alpha_d}\rho^{\gamma_d}\|f\|_2\|g\|_2.
\]
On the other hand,
$\|\sup_{1\le t\le2}|A_tP_\lambda f\, A_tP_\mu g|\|_\infty
\lesssim\lambda\mu\|f\|_1\|g\|_1$, since
$\|\sigma_t*\check P_\ell\|_\infty\lesssim\ell$ uniformly for
$t\in[1,2]$; this follows by integrating the rapidly decreasing kernel
of $P_\ell$ over the sphere of radius $t$. Interpolating these two
estimates with weight $\vartheta_d$ on the first (see
\cref{rem:linearization}), we see that the powers of $\lambda\mu$
cancel, and
\begin{equation}\label{eq:sphere-gap}
 \|\sup_{1\le t\le2}|A_tP_\lambda f\, A_tP_\mu g|\|_{q_d/2}
 \lesssim\rho^{\delta_d}\|f\|_{p_d}\|g\|_{p_d},
\end{equation}
where $\delta_d=\vartheta_d\gamma_d>0$; here we used
$1/p_d=1-\vartheta_d/2$ and $2/q_d=\vartheta_d/r_d$.

Let $d=2$. Expanding $a_{\pm,\ell}/w$ as above, we write the
circular average $A_tP_\ell f$ as a rapidly convergent sum of the
operators $V_\ell$, multiplied by the factors $e^{2\pi imt/T_0}$ and
applied to $a_{\pm,\ell,m}(D)f$. The multipliers
$a_{\pm,\ell,m}(D)$ are bounded on $L^{5/2}$ by
\eqref{eq:amplitude}. By \cref{rem:signs}, since $w$ is real-valued,
we may apply \eqref{eq:planar-model-gap} to each product of two terms,
and we obtain \eqref{eq:sphere-gap} with $\delta_2=c$; by
\cref{prop:planar-gap}, this includes comparable frequencies and
$\lambda=\mu$. With $\lambda=2^j$ and $\mu=2^k$, \eqref{eq:sphere-gap}
is \eqref{intro:eq:spherical-gap}.

\subsection{Summation over frequencies}
We apply \cref{lem:power-sum} with
\[
 (p,q)=(p_d,q_d),\qquad
 m=\begin{cases}3,&d=2,\\2,&d\ge3.\end{cases}
\]
The condition $\max\{2,p\}\le m\le q$ holds, since $5/2\le3\le5$ for
$d=2$, and $1<p_d<2<q_d$ for $d\ge3$. Together with
\eqref{eq:sphere-gap}, this bounds the contribution of the frequencies
$\ell=2^j$, $j\ge1$, by $C\|f\|_{p_d}$. The remaining low-frequency part of
$A_t$ has a kernel bounded by a fixed integrable rapidly decreasing
function, uniformly for $1\le t\le2$, so Young's inequality, with
$1+1/q_d=1/p_d+1/r$, controls its supremum. This proves
\eqref{eq:main} for Schwartz functions. Applying the bound to
differences of Schwartz approximations gives convergence in
$L^{q_d}_x(C([1,2]))$; at every fixed time, the limit agrees with
$A_tf$. Thus the estimate extends by density to $L^{p_d}$.
\begin{remark}[Stein's analytic family]\label{rem:analytic}
For $\alpha\in\mathbb C$ with $\Re\alpha>0$ let
\[
 A^\alpha_tf(x)=\Gamma(\alpha)^{-1}\int_{|y|<1}(1-|y|^2)^{\alpha-1}
 f(x-ty)\,dy,
\]
extend $A^\alpha_t$ to all $\alpha$ by analytic continuation, and let
$\M^\alpha f=\sup_{1\le t\le2}|A^\alpha_tf|$.
With $a=\Re\alpha$, set
\[
 \frac1{p_a}=\frac{d(d-1)+(d+1)a}{d^2+1},\qquad
 \frac1{q_a}=\frac{(d-1)(1-a)}{d^2+1}.
\]
The argument of this section gives
$\|\M^\alpha f\|_{q_a}\lesssim_{d,\alpha}\|f\|_{p_a}$ for $0\le a<1$
when $d=2$, and for $-(d^2-2d-1)/(2(d+1))\le a<1$ when $d\ge3$; the
case $\alpha=0$ recovers \cref{thm:main}, up to a dimensional
normalization.

For fixed $\alpha$, \eqref{eq:bessel} holds with prefactor
$\ell^{-(d-1)/2-\alpha}$ and the same symbol bounds, with
constants depending on $\alpha$. Moreover,
\[
 \sup_{1\le t\le2}
 \|A_t^\alpha P_\ell\|_{L^1\to L^\infty}
 \lesssim_{d,\alpha}\ell^{1-a}.
\]
This follows from angular stationary phase for $|x|\ge1/2$
and radial integration by parts for $|x|<1/2$ in the Bessel
representation, also when $a\le0$.

After the Bessel reduction, we interpolate the bilinear bound
from \eqref{eq:wave-gap-local} for $d\ge3$, or
\eqref{eq:planar-model-gap} for $d=2$, with the kernel bound
above, using weights $(1-a)\vartheta_d$ and $1-a$,
respectively. These weights lie in $(0,1]$ in the stated
ranges; weight one means direct application of the first
bound. The powers of $\lambda\mu$ cancel, and
\cref{rem:linearization} gives
\[
 \left\|\sup_{1\le t\le2}
 |A_t^\alpha P_\lambda f\,A_t^\alpha P_\mu g|\right\|_{q_a/2}
 \lesssim_{d,\alpha}
 \rho^{\delta_a}\|f\|_{p_a}\|g\|_{p_a},
\]
where $\delta_a=(1-a)\vartheta_d\gamma_d>0$ for $d\ge3$
and $\delta_a=(1-a)c>0$ for $d=2$, with $c$ as in
\eqref{eq:planar-model-gap}.

Since $1<p_a\le2<q_a$ for $d\ge3$ and
$1<p_a\le5/2<3<q_a$ for $d=2$,
\cref{lem:power-sum} applies with $m=2$ and $m=3$,
respectively. The low-frequency kernels have a uniform
Schwartz envelope, so Young's inequality and density
complete the proof.
These estimates lie on the boundary of the range of
sufficient conditions in \cite{LSSY2026}.
\end{remark}

\section*{Acknowledgments}
The authors were supported by the National Research Foundation of Korea
under Grant No.~RS-2024-00342160. The authors used GPT (OpenAI) to help develop their ideas, check
calculations, and improve the exposition of the manuscript. The authors
take full responsibility for the mathematical content and the final
manuscript.

\end{document}